\documentclass[10pt]{article}
\usepackage[utf8]{inputenc}
\usepackage[stable]{footmisc}
\usepackage{amsmath}  
\allowdisplaybreaks
\usepackage{graphicx} 
\usepackage{amsthm}
\usepackage{tikz-cd} 
\usepackage{amssymb}
\usepackage{bm}
\usepackage{bbm}
\usepackage{booktabs}
\usepackage{pifont}
\usepackage{xcolor}
\usepackage{array}
\usepackage{makecell}
\usepackage{dsfont}
\usepackage{graphicx}
\usepackage{subcaption}
\usepackage{blindtext}
\usepackage{url}
\usepackage{algorithm}
\usepackage{algorithmic}
\usepackage{thmtools}
\usepackage{thm-restate}
\usepackage{setspace}
\usepackage{enumitem}
\usepackage[margin=1in]{geometry}
\usepackage{xcolor}
\usepackage{cite} 

\usepackage{tcolorbox}
\usepackage{tabularx}
\usepackage{threeparttable}
\usepackage{authblk}
\usepackage{etoc}
\usepackage[export]{adjustbox}
\usepackage{mathabx}

\newtheorem{theorem}{Theorem}[section]
\newtheorem{proposition}[theorem]{Proposition}
\newtheorem{lemma}[theorem]{Lemma}
\newtheorem{corollary}[theorem]{Corollary}
\newtheorem{assumption}[theorem]{Assumption}

\theoremstyle{definition}
\newtheorem{definition}[theorem]{Definition}

\theoremstyle{definition}

\newtheorem{remark}[theorem]{Remark}

\numberwithin{equation}{section}

\newcommand{\R}{\mathbb{R}}
\newcommand{\N}{\mathbb{N}}

\newcommand{\BBS}{\mathbb{S}}

\newcommand{\CP}{\mathcal{P}}
\newcommand{\CN}{\mathcal{N}}
\newcommand{\CX}{\mathcal{X}}

\newcommand{\CA}{\mathcal{A}}
\newcommand{\CS}{\mathcal{S}}

\newcommand{\CT}{\mathcal{T}}

\newcommand{\CF}{\mathcal{F}}

\newcommand{\E}{\mathbb{E}}
\renewcommand{\P}{\mathsf{P}}

\newcommand{\Law}{\mathrm{Law}}

\newcommand{\sgn}[1]{\mathrm{sgn}\left(#1\right)}
\newcommand{\dummy}{\mathord{\color{black!33}\bullet}}

\newcommand{\tr}[1]{\mathrm{tr}(#1)}

\newcommand{\divSphere}{\operatorname{div}_{\mathbb S^{d-1}}}

\newcommand{\SLpm}{\mathrm{SL}^{\pm}}
\newcommand{\Wrho}{\widetilde{\rho}}
\newcommand{\ACG}{\mathrm{ACG}}
\newcommand{\SPD}{\mathrm{SPD}}
\newcommand{\op}{\mathrm{op}}
\newcommand{\Wvarrho}{\widetilde{\varrho}}
\newcommand{\Wx}{\widetilde{x}}
\newcommand{\WX}{\widetilde{X}}
\newcommand{\WY}{\widetilde{Y}}
\newcommand{\WK}{\widetilde{K}}
\newcommand{\Bx}{\mathbf{x}}

\newcommand{\BWx}{\widetilde{\mathbf{x}}}

\newcommand{\BB}{\mathbf{B}}
\newcommand{\WF}{\widetilde{F}}
\newcommand{\Pperp}{\P^{\perp}}
\newcommand{\DPsi}{\dot{\Psi}}
\newcommand{\PsiMu}{\Psi^{\mu}}
\newcommand{\PsiRho}{\Psi^{\rho}}
\newcommand{\Wmu}{\widetilde{\mu}}
\newcommand{\PsiWMu}{\Psi^{\Wmu}}
\newcommand{\PsiWRho}{\Psi^{\Wrho}}
\newcommand{\GL}{\mathrm{GL}}
\newcommand{\adj}{\mathrm{adj}}
\newcommand{\BarRho}{\bar{\rho}}
\newcommand{\PsiBRho}{\Psi^{\BarRho}}
\newcommand{\diag}{\mathrm{diag}}
\newcommand{\WCA}{\widetilde{\CA}}

\makeatletter
\def\url@leostyle{%
	\@ifundefined{selectfont}{\def\UrlFont{\sf}}{\def\UrlFont{\small\ttfamily}}}
\makeatother
\definecolor{darkgreen}{rgb}{0,0.4,0}

\usepackage[final,hypertexnames=false]{hyperref} 
\hypersetup{
	colorlinks=true,       
	linkcolor=blue,        
	citecolor=blue,        
	filecolor=magenta,     
	urlcolor=blue         
}
\usepackage{cleveref}

\newcommand{\Unif}{\mathrm{Unif}}

\title{{\usefont{OT1}{bch}{b}{n}
	\LARGE  An Analytical Framework for a Multi-Particle Oja Flow}}
\author{Sixu Li\thanks{Email: \texttt{sli370@jh.edu}}}
\affil{Department of Mathematics, Johns Hopkins University}
\date{\today}

\begin{document}

\maketitle

\begin{abstract}
We develop an analytical framework for studying a family of interacting particle systems that we refer to as the multi-particle Oja flow. 
The common interaction structure underlying this family arises in a variety of applications, including higher-order synchronization models, opinion dynamics, and self-attention dynamics derived from Transformer architectures, motivating a unified mathematical framework for their analysis.
We first establish well-posedness for both the finite-particle and the mean-field systems. 
We also prove convergence of the finite-particle dynamics to the mean-field equation as the number of particles tends to infinity.
We then investigate low-dimensional structures underlying the dynamics.
In particular, we show that the mean-field solution admits a representation through a time-dependent normalized linear transformation acting on the initial distribution, with the transformation governed by a matrix ODE.
Building on this representation, we identify a broad class of parametric families that are invariant under the mean-field dynamics.
When the initial distribution belongs to one of these families, the evolution remains within the same family for all time and can therefore be characterized through the dynamics of the corresponding family parameters.
As a concrete example, we develop this reduction within the angular central Gaussian family and illustrate how the resulting parameter dynamics provide reduced descriptions of the mean-field evolution for several representative systems.
\end{abstract}

\tableofcontents

\etocdepthtag.toc{main}
\etocsettagdepth{main}{subsection}
\etocsettagdepth{appendix}{none}

\section{Introduction}
Interacting particle systems have long played important roles as mathematical models and as analytical and computational tools across a broad range of problems arising in science \cite{watts1998collective,orsogna2006self,carrillo2010particle,ha2008particle,motsch2014heterophilious}, engineering \cite{pagonabarraga2001dissipative,orsogna2006self,albi2019vehicular,guo2021overviews}, and, more recently, machine learning \cite{liu2016stein,rotskoff2018parameters,kovachki2019ensemble,carrillo2021consensus,geshkovski2025mathematical,el2026physics}. 
However, the mathematical analysis of such systems can be challenging, as the particles depend on one another and evolve collectively through their interactions.
At the same time, many interacting particle systems exhibit symmetries and possess hidden structures that allow their dynamics to be described through a small number of governing quantities.
Uncovering such structures can substantially simplify the analysis of the associated dynamics.
A classical example is the Kuramoto model \cite{kuramoto1975international,kuramoto1984chemical}, whose dynamics admit reduced representations revealed by the Watanabe–-Strogatz transformation \cite{watanabe1994constants,marvel2009identical} and the Ott–Antonsen ansatz \cite{ott2008low,ott2009long}.
These representations have provided powerful tools for analyzing the model and its variants, while also opening new directions for further study \cite{pikovsky2015dynamics,kuramoto2026half}.

In this paper, we consider the following family of interacting particle systems on the unit sphere $\BBS^{d-1}$:
\begin{equation}\label{eq:MP_oja_flow}
    \Dot{x}_k(t) = \left(I_d - x_k(t) x_k(t)^{\top} \right) \CF[\rho_n(t)] x_k(t), \quad x_k(0) = x_{k,0},  \qquad \text{for } k = 1,2, \dots, n \, .
\end{equation}
Here, $\rho_n(t) := \frac{1}{n} \sum_{k=1}^n \delta_{x_k(t)}$ is the empirical measure associated with the particle positions, and the interaction kernel $\CF: \CP(\BBS^{d-1}) \to \R^{d\times d}$ is a matrix-valued functional defined on probability measures over $\BBS^{d-1}$.
For later use, we denote the associated velocity field by
\begin{equation}\label{eq:vf}
    b (x, \varrho) := \left( I_d - xx^{\top} \right) \CF[\varrho] x, \qquad x \in \BBS^{d-1}, \; \varrho \in \CP(\BBS^{d-1}) \, .
\end{equation}

The system \eqref{eq:MP_oja_flow} can be viewed as a multi-particle Oja-type model.
Indeed, when the interaction kernel $\CF$ is a fixed matrix, each particle evolves independently according to the classical Oja flow \cite{oja1982simplified}; see Section \ref{sec:examples} for a further illustration.
We therefore refer to this family of interacting particle systems as the \emph{multi-particle Oja flow}.

For suitable choices of the interaction kernel $\CF$, the multi-particle Oja flow \eqref{eq:MP_oja_flow} encompasses
a variety of systems arising in different applications.
These include higher-dimensional Kuramoto-type models \cite{dai2021dimensional,kim2021cluster,lohe2025exactA}, opinion dynamics \cite{zhang2021dissensus,zhang2022opinion}, and linear self-attention dynamics \cite{li2026diverse}; see Section \ref{subsec:related_work} for a more detailed discussion. 
Different choices of $\CF$ can give rise to qualitatively distinct collective behaviors. 
Figure \ref{fig:oja-overview} shows the evolution of \eqref{eq:MP_oja_flow} under two different interaction kernels.

\begin{figure*}[!htb]
    \centering

    \begin{subfigure}{\textwidth}
        \centering
        \includegraphics[width=\textwidth]{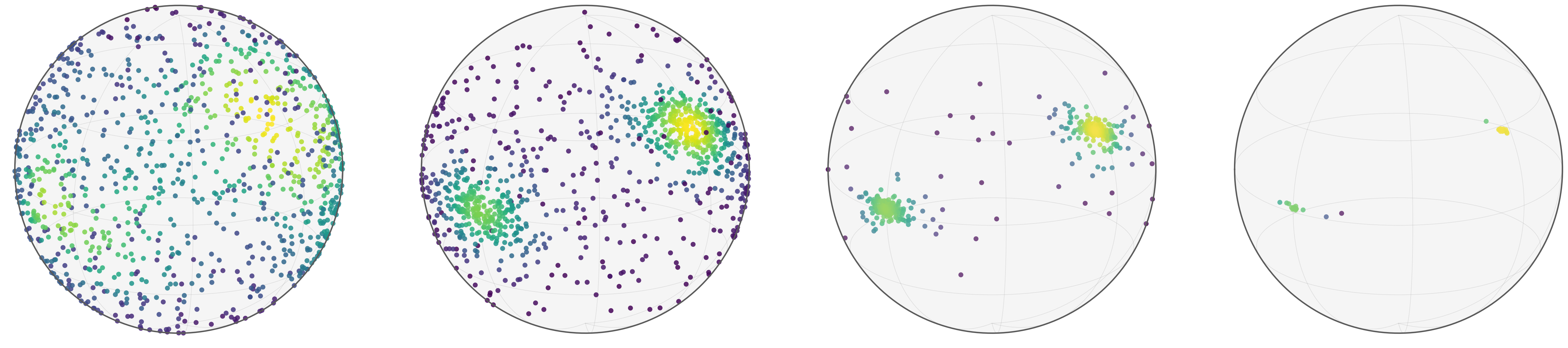}
        \caption{Second-moment interaction kernel \eqref{eq:CF_second_moment}.}
        \label{fig:oja-second-moment}
    \end{subfigure}

    \vspace{0.6em}

    \begin{subfigure}{\textwidth}
        \centering
        \includegraphics[width=\textwidth]{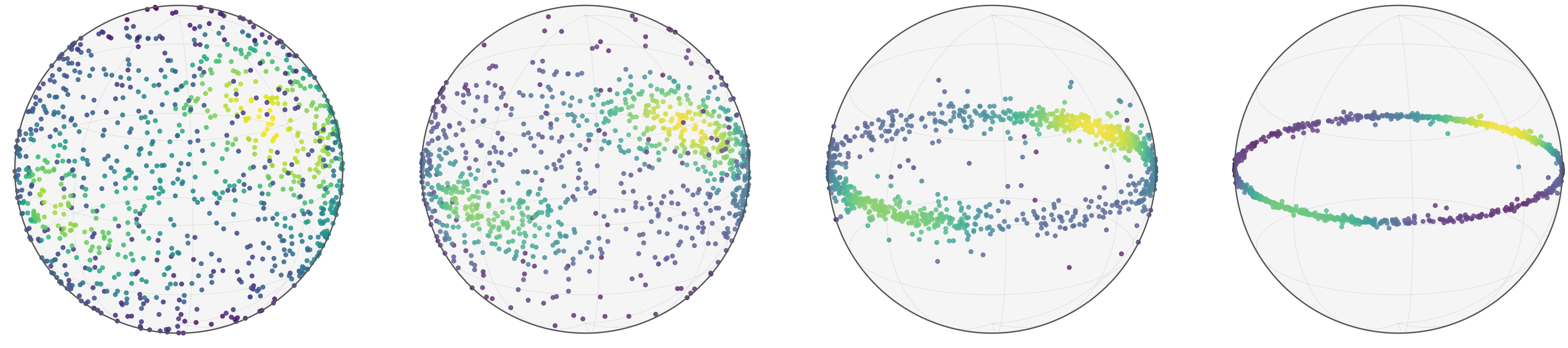}
        \caption{Cross-product interaction kernel \eqref{eq:CF_three_body}.}
        \label{fig:oja-three-body}
    \end{subfigure}

    \caption{
    Evolution of the multi-particle Oja flow \eqref{eq:MP_oja_flow} under two
    different interaction kernels. 
    Each row shows snapshots of the particle
    configuration on $\mathbb{S}^2$ at four successive times.
    The color of each particle represents the estimated local particle density, with purple and yellow corresponding to lower and higher densities, respectively.
    }
    \label{fig:oja-overview}
\end{figure*}

The recurrence of the same dynamical form across different settings and the variety of collective behaviors exhibited by these systems motivate a systematic study of the multi-particle Oja flow \eqref{eq:MP_oja_flow} as a general class of interacting particle systems.
To set the stage for the developments that follow, we first introduce the corresponding mean-field dynamics.

Suppose that the initial empirical measures $\rho_n(0) = \frac{1}{n} \sum_{k=1}^n \delta_{x_{k, 0}}$ converge, as $n \rightarrow \infty$, to a probability measure $\rho_0 \in \CP(\BBS^{d-1})$.
Formally, the mean-field limit of \eqref{eq:MP_oja_flow} is described by the continuity equation on the sphere $\BBS^{d-1}$:
\begin{equation}\label{eq:mf_oja_flow}
\partial_t\rho
+
\divSphere \left( b(\dummy, \rho)\rho \right)
=0, \qquad \rho(0) = \rho_0,
\end{equation}
where $\divSphere$ denotes the spherical divergence operator and the vector field $b$ is defined in \eqref{eq:vf}.
The corresponding characteristic dynamics can be interpreted as the evolution of a representative particle:
\begin{equation}\label{eq:representative_particle_ode}
    \dot{X}(t) = b(X(t), \rho(t)), \qquad X(0) = X_0,
\end{equation}
where $X_0 \sim \rho_0$ and $\rho(t) = \Law(X(t))$.

\subsection{Contributions}
In this paper, we develop an analytical framework for studying the multi-particle Oja flow.
As a first step, we establish the well-posedness of the dynamics at both the finite-particle and mean-field levels, together with a mean-field convergence result.
We then investigate hidden \emph{low-dimensional structures} underlying the dynamics.\footnote{Here, by low-dimensional structures we mean properties of the dynamics that allow the particle system or mean-field evolution to be described by time-dependent quantities whose dimension does not grow with the number of particles \cite{watanabe1994constants,ott2008low,ott2009long,marvel2009identical}.} 
Our contributions are summarized as follows.
\begin{itemize}
    \item[(i)] \textbf{Well-posedness and mean-field convergence.} 
    We establish the well-posedness of both the finite-particle system \eqref{eq:MP_oja_flow} and the mean-field equation \eqref{eq:mf_oja_flow}; see Theorem \ref{thm:well-posedness_finite_sys} and Theorem \ref{thm:well-posedness_mf}, respectively.
    In Theorem \ref{thm:well-posedness_mf}, we also prove Wasserstein stability of the mean-field solution with respect to the initial distribution. 
    Combined with the empirical-measure representation of the finite-particle system, this stability result gives the convergence of the associated empirical measures to the mean-field solution stated in Corollary \ref{cor:mf_app}.

    \item[(ii)] \textbf{Representation via normalized linear transformations.} 
    We show that the solution to the mean-field equation \eqref{eq:mf_oja_flow} admits an explicit representation as the pushforward of the initial distribution through a time-dependent normalized linear transformation as defined in \eqref{eq:NL_transform}.
    The evolution of this transformation is governed by a matrix ODE; see Theorem \ref{thm:PI_mf_sys}.

    \item[(iii)] \textbf{Reduction via invariant parametric families.}
    Building on this representation, we identify a broad class of parametric families that are invariant under the mean-field equation \eqref{eq:mf_oja_flow}.
    That is, if the initial distribution belongs to one of these families, then the solution remains within the same family for all time, reducing the mean-field evolution to the dynamics of the corresponding family parameters.
    A general approach for identifying such invariant parametric families is presented in Proposition \ref{prop:construct_inv}.
    
    As an important example, we show that the angular central Gaussian family is invariant under the mean-field equation \eqref{eq:mf_oja_flow}.
    Moreover, we derive the evolution equation for its covariance parameter matrix, providing an explicit description of the mean-field dynamics within this family; see Theorem \ref{thm:ACG_invariant}.
\end{itemize}

\subsection{Related Work}\label{subsec:related_work}
\paragraph{Related models and applications.}
Dynamics that can be recast in the form of the multi-particle Oja flow \eqref{eq:MP_oja_flow} have appeared in several settings.
In particular, several higher-dimensional generalizations of the Kuramoto model with higher-order interactions \cite{dai2021dimensional,kim2021cluster,lohe2025exactA} fall within this class.
Related dynamics also arise from the double-sphere model introduced in \cite{lohe2020on}, which was motivated by a reformulation of a non-Abelian Kuramoto model \cite{lohe2009non}.
Opinion dynamics models designed to reach dissensus states also provide instances of the multi-particle Oja flow \cite{zhang2021dissensus,zhang2022opinion}.
A further connection can be identified with linear Transformer architectures \cite{katharopoulos2020transformers,peng2021random,schlag2021linear}:
by abstracting and idealizing the linear self-attention mechanism as in \cite{li2026diverse}, the associated dynamics can likewise be cast within this family.
A related structure also arises in the Doi--Onsager equation with the Maier--Saupe potential \cite{liu2005axial}, whose noiseless dynamics belong to the same class.

Table \ref{tab:interaction-kernels} summarizes representative interaction kernels associated with the systems described above when written in the form of the multi-particle Oja flow; some systems involve scaled versions or combinations of the kernels listed there.

\begin{table*}[!htb]
    \centering
    \small
    \renewcommand{\arraystretch}{1.5}
    \caption{
    Representative interaction kernels obtained by rewriting systems arising in different applications in the form 
    \eqref{eq:MP_oja_flow}.
    The kernel names are adopted here for convenience.
    }
    \label{tab:interaction-kernels}

    \begin{tabularx}{0.9\textwidth}{@{
        }>{\raggedright\arraybackslash}p{0.28\textwidth}
        >{\raggedright\arraybackslash}X
        >{\centering\arraybackslash}p{0.15\textwidth}@{}}
        \toprule
        Name
        & Formula
        & References
        \\
        \midrule

        Mean outer-product kernel
        &
        $\displaystyle
        \CF[\varrho]
        =
        \left(
        \int_{\BBS^{d-1}} x\,d\varrho(x)
        \right)
        \left(
        \int_{\BBS^{d-1}} x\,d\varrho(x)
        \right)^{\top}
        $
        &
        \cite{dai2021dimensional}
        \\[1.0em]

        Second-moment kernel
        &
        $\displaystyle
        \CF[\varrho]
        =
        \int_{\BBS^{d-1}} xx^{\top}\,d\varrho(x)
        $
        &
        \cite{liu2005axial,lohe2020on,kim2021cluster,zhang2022opinion,lohe2025exactA}
        \\[1.0em]

        Cross-product kernel
        &
        $\displaystyle
        \CF[\varrho]
        =
        -\int_{\BBS^2}\int_{\BBS^2}
        (x\times y)(x\times y)^{\top}
        \,d\varrho(x)\,d\varrho(y)
        $
        &
        \cite{lohe2025exactA}
        \\[1.0em]

        Linear self-attention kernel
        &
        $\displaystyle
        \CF[\varrho]
        =
        V
        \left(
        \int_{\BBS^{d-1}} xx^{\top}\,d\varrho(x)
        \right)
        A, \quad A,V \in \R^{d\times d}
        $
        &
        \cite{li2026diverse}
        \\

        \bottomrule
    \end{tabularx}
\end{table*}

\paragraph{Related low-dimensional reductions.}
The closest works to the present paper are \cite{lohe2020on,lohe2025exactA}, where the system \eqref{eq:MP_oja_flow} was previously studied under the name \emph{cubic system}, referring to the cubic dependence of the associated vector field $b$ in \eqref{eq:vf} on the state variable $x$.
At the finite-particle level, these works show that all particle trajectories can be represented through a common time-dependent normalized linear transformation applied to the initial configuration.

Building on the finite-particle results of \cite{lohe2020on,lohe2025exactA}, the present work extends the normalized linear transformation representation to the mean-field setting, with the previous finite-particle result recovered as a special case; see Remark \ref{rem:existing_finite_sys_result} for details.
Importantly, the distribution-level formulation enables a further reduction that is specific to the mean-field setting. 
In particular, we identify invariant parametric families within which the mean-field evolution can be reduced to the dynamics of the corresponding family parameters.
For initial data in one of these families, this reduction gives a more explicit description of the evolving distribution.

These two types of reductions are closely related in spirit to the Watanabe--Strogatz theory \cite{watanabe1994constants,marvel2009identical,Lohe_2018,Lipton_2021} and the Ott--Antonsen ansatz \cite{ott2008low,ott2009long,chandra2019complexity} for the Kuramoto model \cite{kuramoto1975international,kuramoto1984chemical}.
The former shows that the Kuramoto dynamics admit a representation through a time-dependent Möbius transformation acting on the initial data, while the latter identifies a particular parametric family of distributions that is invariant under the associated mean-field dynamics.
The parallels between our results and the Watanabe--Strogatz theory and the Ott--Antonsen ansatz are discussed in greater detail in Remark \ref{rem:connect_to_WS} and Remark \ref{rem:connect_to_OA}, respectively.

\paragraph{The classical Oja flow and multi-particle Oja-type models.}
The classical Oja flow was originally derived from a modified Hebbian learning rule\footnote{Hebbian learning, originating in neuropsychology, refers to the principle that connections between neurons are strengthened through correlated activity \cite{hebb1949organization}.} for the weights of a linear neuron \cite{oja1982simplified}. 
The resulting dynamics provide a mechanism for principal component extraction and have since been extensively studied for principal component analysis \cite{oja1985stochastic,oja1989neural,oja1992principal,yan1994global,chen1998global,tsuzuki2025global}.

Multi-particle Oja-type models have also appeared in the literature.
In particular, \cite{zhang2021dissensus,zhang2022opinion} introduce Oja-type nonlinear opinion dynamics on the unit sphere in which the interactions among agents are determined by the covariance matrix or the second moment matrix of the opinion states.
These dynamics fall within the family of multi-particle Oja flows of the form \eqref{eq:MP_oja_flow}.
A different multi-particle Oja-type model, motivated by continuous-time self-attention dynamics for Transformer architectures, is developed in \cite{altafini2026multistability}.
In particular, \cite{altafini2026multistability} introduces a model termed the \emph{multiagent Oja flow} (see equation (5) therein), obtained by replacing the state-dependent attention weights in the softmax self-attention dynamics with uniform coupling.
This model is distinct from the multi-particle Oja flow \eqref{eq:MP_oja_flow} considered in the present paper.
Rather, its structure is closely related to higher-dimensional generalizations of the Kuramoto model of the type considered in \cite{chandra2019complexity,Lipton_2021}.

\paragraph{Notation.}
We use $\dummy$ as a placeholder for the argument of a function.
We write $\|\dummy\|$ for the Euclidean norm of a vector, and $\|\dummy\|_{\op}$ for the matrix operator norm.
$\CP(\BBS^{d-1})$ denotes the set of probability measures on $\BBS^{d-1}$, equipped with the Wasserstein-$1$ distance $W_1$ induced by the Euclidean distance $\|\dummy\|$ on $\R^d$.
For a measurable map $h$, $h_{\sharp}\mu$ denotes the pushforward of a measure $\mu$ under $h$.
For a square matrix $G$, we denote its adjugate by $\adj(G)$.

\section{Well-posedness}\label{sec:well-posedness}
In this section, we study the well-posedness of the multi-particle Oja flow.
We first consider the finite-particle system \eqref{eq:MP_oja_flow} in Section \ref{subsec:well-posedness_finite_sys} and then turn to the mean-field equation \eqref{eq:mf_oja_flow} in Section \ref{subsec:well-posedness_mf}, where we also establish a mean-field convergence result.

We begin by introducing the assumption on the interaction kernel $\CF$ required for our analysis and recording some basic properties of the associated vector field $b$.

\begin{assumption}\label{asm:well-posedness}
There exists a constant $L_{\CF} \geq 0$ such that, for every $\varrho, \Wvarrho \in \CP(\BBS^{d-1})$,
\begin{equation*}
    \left\| \CF[\varrho] - \CF[\Wvarrho] \right\|_{\op} \leq L_{\CF} W_1 (\varrho, \Wvarrho) \, .
\end{equation*}
\end{assumption}
In other words, we assume that $\CF$ is Lipschitz continuous with respect to the Wasserstein-$1$ distance on $\CP(\BBS^{d-1})$.
All interaction kernels in Table \ref{tab:interaction-kernels} satisfy this assumption. 
Moreover, since $\BBS^{d-1}$ is compact, $\CP(\BBS^{d-1})$ is compact with respect to $W_1$. 
The Lipschitz continuity of $\CF$ therefore implies that it is uniformly bounded, namely,
\begin{equation*}
C_{\CF} := \sup_{\varrho \in \CP(\BBS^{d-1})} \|\CF [\varrho]\|_{\op} < \infty \, .
\end{equation*}

The boundedness and Lipschitz continuity of $\CF$ give corresponding uniform boundedness and Lipschitz estimates for the velocity field associated with the multi-particle Oja flow.
These properties are collected in the following lemma, and the associated proof is provided in Appendix \ref{app:proof_sec_2}.

\begin{lemma}\label{lem:vf_property}
Suppose that Assumption \ref{asm:well-posedness} holds.
The vector field $b$ defined in \eqref{eq:vf} satisfies the following properties:
\begin{itemize}
    \item For every $x \in \BBS^{d-1}$ and $\varrho \in \CP(\BBS^{d-1})$, 
    \[
    \|b(x, \varrho)\| \leq C_{\CF} \, .
    \]

    \item For every $x, \Wx \in \BBS^{d-1}$ and every $\varrho, \Wvarrho \in \CP(\BBS^{d-1})$, 
    \begin{equation*}
    \left\| b(x, \varrho) - b (\Wx, \Wvarrho) \right\| \leq 3 C_{\CF} \|x - \Wx\| + L_{\CF} W_1 (\varrho, \Wvarrho) \, .
    \end{equation*}
\end{itemize}
\end{lemma}
With these preliminary estimates in place, we are ready to establish the well-posedness results for the multi-particle Oja flow.

\subsection{Well-posedness of the Finite-particle System}\label{subsec:well-posedness_finite_sys}
We begin by showing the existence and uniqueness of global solutions to the finite-particle system \eqref{eq:MP_oja_flow}.

\begin{theorem}[Well-posedness of the finite-particle system]\label{thm:well-posedness_finite_sys}
Suppose that Assumption \ref{asm:well-posedness} holds.
Then, for every $n \in \N$, the system \eqref{eq:MP_oja_flow} admits a unique global solution $(x_1, x_2, \dots, x_n) \in C^1 ([0, \infty); (\BBS^{d-1})^n)$ for any initial configuration in $(\BBS^{d-1})^n$.
\end{theorem}

The main idea of the proof is to show that the vector field corresponding to \eqref{eq:MP_oja_flow} is uniformly bounded, globally Lipschitz, and tangent to the product manifold $(\BBS^{d-1})^n$.
Standard existence and uniqueness results for ODEs on manifolds then give a unique local solution, while the compactness of the state space ensures that the solution extends globally in time; see, for example, \cite{hartman2002ordinary}.

\begin{proof}
Fix $n \in \N$.
For a particle configuration $\Bx := \big(x_1, \dots, x_n\big)^{\top} \in (\BBS^{d-1})^{n}$, define the associated empirical distribution by $\rho_n^{\Bx} = \frac{1}{n} \sum_{k=1}^n \delta_{x_k}$.
For each $k=1,2,\dots, n$, set
\begin{equation*}
B_k (\Bx) := b \big(x_k, \rho_n^{\Bx} \big),
\end{equation*}
and define the full vector field
\begin{equation*}
    \BB (\Bx) := \left( B_1 (\Bx), \dots, B_n (\Bx)\right)^{\top} \in \R^{nd} \, .
\end{equation*}
The finite-particle system \eqref{eq:MP_oja_flow} can then be written as
\begin{equation}\label{eq:MP_oja_flow_rewrite}
\dot{\Bx}(t) = \BB (\Bx(t) ) \, .
\end{equation}

We equip $(\BBS^{d-1})^n$ with the metric
\begin{equation*}
    d_{\infty} (\Bx, \BWx) := \max_{1 \leq k \leq n} \|x_k - \Wx_k\| \, .
\end{equation*}
For any two particle configurations $\Bx, \BWx \in (\BBS^{d-1})^n$, the coupling that pairs $x_k$ with $\Wx_k$ gives
\begin{equation*}
    W_1 (\rho_n^{\Bx}, \rho_n^{\BWx}) \leq \frac{1}{n} \sum_{k=1}^n \|x_k - \Wx_k\| \leq d_{\infty} (\Bx, \BWx) \, .
\end{equation*}
By Lemma \ref{lem:vf_property}, for every $k=1,\ldots, n$,
\begin{equation*}
    \|B_k(\Bx)\| = \|b(x_k, \rho_n^{\Bx})\| \leq C_{\CF},
\end{equation*}
and
\begin{equation*}
\|B_k (\Bx) - B_k (\BWx)\| = \|b(x_k, \rho_n^{\Bx}) - b(\Wx_k, \rho_n^{\BWx})\| \leq 3C_{\CF} \|x_k - \Wx_k\| + L_{\CF} W_1 (\rho_n^{\Bx}, \rho_n^{\BWx}) \leq \left(3C_{\CF} + L_{\CF} \right) d_{\infty} (\Bx, \BWx) \, .
\end{equation*}
Moreover, a straightforward calculation shows that $x_k^{\top} B_k (\Bx) = 0$, and hence $\BB(\Bx)$ is tangent to $(\BBS^{d-1})^n$ at $\Bx$.
Therefore, the vector field $\BB$ of the system \eqref{eq:MP_oja_flow_rewrite} is uniformly bounded, globally Lipschitz, and tangent to the product manifold $(\BBS^{d-1})^n$.

Then, standard existence and uniqueness results for ODEs on manifolds imply that, for every initial configuration in $(\BBS^{d-1})^n$, \eqref{eq:MP_oja_flow_rewrite} admits a unique maximal solution.
Finally, since $(\BBS^{d-1})^n$ is compact, the maximal solution cannot escape the state space in finite time.
It therefore extends to all $t \geq 0$, which completes the proof.
\end{proof}

\subsection{Well-posedness of the Mean-field Equation}\label{subsec:well-posedness_mf}
We begin by recalling the definition of a weak solution.
A curve $\rho \in C([0,\infty); \CP(\BBS^{d-1}))$ is called a weak solution to \eqref{eq:mf_oja_flow} with initial condition $\rho_0$ if $\rho(0) = \rho_0$ and, for every test function $\varphi \in C^1 (\BBS^{d-1})$ and every $t \geq 0$,
\begin{equation}
\begin{aligned}
\int_{\BBS^{d-1}} \varphi (x) d\rho(t, x) - \int_{\BBS^{d-1}} \varphi (x) d\rho_0(x) &= \int_0^{t} \int_{\BBS^{d-1}}  \nabla_{\BBS^{d-1}} \varphi(x)^{\top} b(x, \rho(s))  d\rho(s,x) ds \, .
\end{aligned}
\end{equation}

\begin{theorem}[Well-posedness of the mean-field equation]\label{thm:well-posedness_mf}
Suppose that Assumption \ref{asm:well-posedness} holds.
Then, for every initial distribution $\rho_0 \in \CP(\BBS^{d-1})$, there exists a unique weak solution $\rho \in C ([0, \infty); \CP(\BBS^{d-1}))$ to the mean-field equation \eqref{eq:mf_oja_flow}.
Moreover, the solution admits the characteristic representation
\begin{equation}\label{eq:pushforward_repre}
\rho(t) = (\Psi_t)_{\sharp} \rho_0,
\end{equation}
where $\Psi_t: \BBS^{d-1} \to \BBS^{d-1} $ is the characteristic flow associated with $\rho$, determined by
\begin{equation}\label{eq:charac_flow}
\frac{d}{dt}\Psi_t (x) = b (\Psi_t(x), \rho(t)), \qquad \Psi_0(x) = x \, .
\end{equation}

The solution is also stable with respect to the initial data in the Wasserstein-$1$ distance.
More specifically, let $\rho$ and $\Wrho$ be the solutions corresponding to initial data $\rho_0$ and $\Wrho_0$, respectively. Then, for every $t \geq 0$,
\begin{equation}\label{eq:ws_stability}
W_1 (\rho(t), \Wrho(t)) \leq e^{(3C_{\CF} + L_{\CF})t} W_1 (\rho_0, \Wrho_0) \, .
\end{equation}
\end{theorem}
The proof of Theorem \ref{thm:well-posedness_mf} is adapted from the classical fixed-point argument for measure-dependent characteristic flows \cite{dobrushin1979valsov,canizo2011well}.
We provide the details at the end of this section.
Before turning to the proof, we discuss some interpretations and consequences of the theorem.

\paragraph{Representative-particle interpretation.}
The characteristic representation  \eqref{eq:pushforward_repre} admits a natural probabilistic interpretation.
Let $X_0$ be a random variable with law $\rho_0$, and define 
\[
X(t) := \Psi_t(X_0) \, .
\]
Then $\Law(X(t)) = (\Psi_t)_{\sharp} \rho_0 = \rho(t)$, and $X(t)$ satisfies the characteristic equation \eqref{eq:representative_particle_ode}.
Thus, $X(t)$, obtained by transporting the random initial state $X_0$ through the characteristic flow $\Psi_t$, may be interpreted as the state of a representative particle in the mean-field limit.

\paragraph{Consistency of the particle and mean-field descriptions.}
The characteristic representation \eqref{eq:pushforward_repre} also shows that the finite-particle system \eqref{eq:MP_oja_flow} is consistent with the mean-field equation \eqref{eq:mf_oja_flow}.
More precisely, consider an empirical initial distribution $\rho_n(0) = \frac{1}{n} \sum_{k=1}^n \delta_{x_{k,0}}$, and let $\rho_n$ be the unique solution to the mean-field equation with this initial condition.
Denote its associated characteristic flow by $\Psi_t^n$.
Then the characteristic representation of $\rho_n$ is given by
\begin{equation}
\rho_n(t) = (\Psi_t^n)_{\sharp} \rho_n(0) = \frac{1}{n} \sum_{k=1}^n \delta_{\Psi_t^n (x_{k,0})} \, .
\end{equation}
Define $x_k(t) := \Psi_t^n (x_{k,0})$ for $k=1,2, \dots, n$. 
Then the characteristic flow \eqref{eq:charac_flow} gives
\begin{equation*}
\dot{x}_k(t) = b (x_k(t), \rho_n(t)) = (I_d - x_k(t) x_k(t)^{\top}) \CF[\rho_n(t)] x_k(t) \, .
\end{equation*}
Therefore, $(x_1, \dots, x_n)$ solves the finite-particle system \eqref{eq:MP_oja_flow}.
By the uniqueness result established in Theorem \ref{thm:well-posedness_finite_sys}, it is precisely the unique solution to that system.
Thus, the finite-particle dynamics are recovered exactly by restricting the mean-field equation to empirical measures.

Combining this consistency property with the stability estimate \eqref{eq:ws_stability}, we obtain the following mean-field convergence result.

\begin{corollary}[Mean-field convergence]\label{cor:mf_app}
Suppose that Assumption \ref{asm:well-posedness} holds.
For each $n \geq 1$, let $\rho_n(0) = \frac{1}{n} \sum_{k=1}^n \delta_{x_{k,0}}$ be an empirical initial distribution, and let $\rho_n(t)$ denote the empirical measure associated with the solution of the finite-particle system \eqref{eq:MP_oja_flow}.
Let $\rho(t)$ be the solution of the mean-field equation \eqref{eq:mf_oja_flow} with initial condition $\rho(0) \in \CP(\BBS^{d-1})$.
Then, for every $t \geq 0$,
\begin{equation}\label{eq:mf_approximation}
W_1 (\rho_n(t), \rho(t)) \leq e^{(3C_{\CF} + L_{\CF})t} W_1 (\rho_n(0), \rho(0)) \, . 
\end{equation}
Hence, as $n \rightarrow \infty$, whenever $W_1(\rho_n(0), \rho(0)) \rightarrow 0$, we have $\sup_{0 \leq t \leq T} W_1 (\rho_n(t), \rho(t)) \rightarrow 0$ for every finite $T > 0$.
\end{corollary}

\begin{proof}

Applying the Wasserstein stability estimate \eqref{eq:ws_stability} to $\rho_n$ and $\rho$, we obtain \eqref{eq:mf_approximation}.
Taking the supremum over $t \in [0,T]$ gives the claimed convergence.
\end{proof}


\begin{remark}
The estimate \eqref{eq:mf_approximation} deteriorates exponentially as $t$ increases and therefore gives convergence only on bounded time intervals.
Establishing uniform-in-time mean-field convergence, possibly under additional assumptions on $\CF$, is left for future work.
\end{remark}

We now conclude this section with the proof of Theorem \ref{thm:well-posedness_mf}.

\begin{proof}[Proof of Theorem \ref{thm:well-posedness_mf}]
We first establish the existence and uniqueness of a weak solution.
The argument is divided into two steps.

\paragraph{Step 1: Construction of a solution by a fixed-point argument.} 
Fix $T > 0$.
Given a prescribed measure-valued curve $\mu \in C([0,T]; \CP(\BBS^{d-1}))$ satisfying $\mu(0) = \rho_0$,
consider the characteristic equation
\begin{equation}\label{eq:chara_flow_in_proof}
    \DPsi^{\mu}_t (x) = b (\PsiMu_t (x), \mu(t)), \qquad \PsiMu_0(x) = x,
\end{equation}
where $b$ is given in \eqref{eq:vf}.
Since $t \mapsto \mu(t)$ is continuous with respect to $W_1$ and $\CF$ is Lipschitz continuous under Assumption \ref{asm:well-posedness}, the map $t \mapsto \CF[\mu(t)]$ is continuous.
Together with Lemma \ref{lem:vf_property}, this implies that the velocity field $(t, x) \mapsto b(x, \mu(t))$ is continuous in time, globally Lipschitz in space, and uniformly bounded.
Moreover, the velocity field is tangent to $\BBS^{d-1}$.
Standard existence and uniqueness theory for ODEs on manifolds therefore implies that \eqref{eq:chara_flow_in_proof} admits a unique solution $\PsiMu_t(x)$ on $[0,T]$; see, for example, \cite{hartman2002ordinary}.

Define the map $\CT : C([0,T]; \CP(\BBS^{d-1})) \to C ([0,T]; \CP(\BBS^{d-1}))$ by
\begin{equation*}
    (\CT \mu)(t) := (\PsiMu_t)_{\sharp} \rho_0 \, .
\end{equation*}
Then a fixed point $\mu = \CT \mu$ satisfies
\begin{equation*}
    \mu(t) = (\PsiMu_t)_{\sharp} \rho_0,
\end{equation*}
and is therefore transported by the characteristic flow generated by itself.
We now show that $\CT$ is a contraction with respect to a suitable metric.

Let $\mu, \Wmu \in C([0,T]; \CP(\BBS^{d-1}))$.
For a fixed $x \in \BBS^{d-1}$, set
\begin{equation*}
    X(t, x) := \PsiMu_t(x), \qquad \WX(t, x) := \PsiWMu_t(x) \, .
\end{equation*}
Since $X(0, x) = \WX(0, x) = x$, Lemma \ref{lem:vf_property} gives
\begin{equation*}
\big\|X(t, x) - \WX(t, x) \big\| \leq 3 C_{\CF} \int_0^t \big\|X(s, x) - \WX(s, x) \big\| ds + L_{\CF} \int_0^t W_1 (\mu(s), \Wmu(s))ds \, .
\end{equation*}
By Grönwall's inequality,
\begin{equation*}
    \big\| X(t, x) - \WX(t, x) \big\| \leq L_{\CF} \int_0^t e^{3C_{\CF}(t-s)} W_1 (\mu(s), \Wmu(s)) ds \, .
\end{equation*}
Since $(\PsiMu_t, \PsiWMu_t)_{\sharp} \rho_0$ is a coupling between $(\CT \mu)(t)$ and $(\CT \Wmu)(t)$, it follows that
\begin{equation}\label{eq:contraction_proof_aux}
W_1 ( (\CT \mu)(t), (\CT \Wmu) (t) ) \leq L_{\CF} \int_0^t e^{3C_{\CF} (t-s)} W_1 (\mu(s), \Wmu(s)) ds \, .
\end{equation}
For $\lambda > 0$, introduce the weighted metric
\begin{equation*}
d_{\lambda} (\mu, \Wmu) := \sup_{0\leq t \leq T} e^{-\lambda t} W_1 (\mu(t), \Wmu(t)) \, .
\end{equation*}
Multiplying the preceding estimate \eqref{eq:contraction_proof_aux} by $e^{-\lambda t}$ and using the definition of $d_{\lambda}$, we obtain, for $\lambda > 3C_{\CF}$,
\begin{equation*}
e^{-\lambda t} W_1 ( (\CT \mu)(t), (\CT \Wmu) (t) ) \leq L_{\CF} d_{\lambda} (\mu, \Wmu) \int_0^t e^{-(\lambda - 3C_{\CF})(t-s)} ds \leq \frac{L_{\CF}}{\lambda - 3 C_{\CF}} d_{\lambda} (\mu, \Wmu)\, .
\end{equation*}
Taking the supremum over $t \in [0,T]$ gives
\begin{equation*}
d_{\lambda} (\CT \mu, \CT \Wmu) \leq \frac{L_{\CF}}{\lambda - 3 C_{\CF}} d_{\lambda} (\mu, \Wmu) \, .
\end{equation*}
Choosing $\lambda > 3 C_{\CF} + L_{\CF}$ makes $\CT$ a strict contraction.

Because $\CP(\BBS^{d-1})$ is complete under $W_1$, the space $C([0,T]; \CP(\BBS^{d-1}))$ is complete under the weighted metric $d_{\lambda}$.
The Banach fixed-point theorem \cite{banach1922sur} therefore gives a unique fixed point $\rho \in C([0,T]; \CP(\BBS^{d-1}))$, satisfying 
\begin{equation*}
\rho(t) = (\Psi_t^{\rho})_{\sharp} \rho_0 \, .
\end{equation*}
Since $T > 0$ is arbitrary and, by uniqueness, solutions constructed on different time intervals agree on their overlaps, they define a global characteristic solution $\rho \in C([0,\infty); \CP(\BBS^{d-1}))$.

\paragraph{Step 2: Verification of the weak formulation.}
Next, we verify that the fixed point $\rho$ constructed in Step $1$ is the unique weak solution to the mean-field equation \eqref{eq:mf_oja_flow}.

For every test function $\varphi \in C^1 (\BBS^{d-1})$ and every $t \geq 0$, the definition of the pushforward measure gives
\begin{equation*}
\int_{\BBS^{d-1}} \varphi (x) d\rho(t, x) = \int_{\BBS^{d-1}} \varphi (\PsiRho_t(x)) d\rho_0(x) \, .
\end{equation*}
Differentiating this identity with respect to $t$ and using the characteristic equation \eqref{eq:chara_flow_in_proof}, we obtain
\begin{equation}\label{eq:verify_weak_form}
\frac{d}{dt} \int_{\BBS^{d-1}} \varphi(x) d\rho(t,x) = \int_{\BBS^{d-1}} \nabla_{\BBS^{d-1}} \varphi( \PsiRho_t(x))^{\top} b (\PsiRho_t(x), \rho(t)) d\rho_0(x) = \int_{\BBS^{d-1}} \nabla_{\BBS^{d-1}} \varphi (x)^{\top} b(x, \rho(t)) d\rho(t, x) \, .  
\end{equation}
Integrating the above identity over $[0,t]$ shows that $\rho$ satisfies the weak formulation of  \eqref{eq:mf_oja_flow} for every $t \geq 0$.
Hence, $\rho$ is a weak solution of the mean-field equation.

It remains to prove uniqueness.
Let $\BarRho$ be any weak solution to \eqref{eq:mf_oja_flow} with initial condition $\rho_0$, and let $\PsiBRho$ denote the characteristic flow \eqref{eq:chara_flow_in_proof}  generated by the curve $\BarRho$.
Once $\BarRho$ is fixed, the map $(t,x) \mapsto b(x, \BarRho(t))$ can be regarded as a given time-dependent velocity field.
By Lemma \ref{lem:vf_property}, this velocity field is continuous in time and globally Lipschitz in space.
The classical characteristic representation for continuity equations with Lipschitz velocity fields therefore implies that $\BarRho(t) = (\PsiBRho_t)_{\sharp} \rho_0$; see, for example, \cite{ambrosio2014continuity}.
Therefore, $\BarRho = \CT \BarRho$, which shows that every weak solution of \eqref{eq:mf_oja_flow} is a fixed point of $\CT$.
By the uniqueness of the fixed point established in Step $1$, we conclude that $\BarRho = \rho$, which proves uniqueness of the weak solution.

\paragraph{Stability Estimate.}
Lastly, we perform the stability estimate of the solution with respect to the initial data.
Let $\rho$ and $\Wrho$ be solutions corresponding to initial data $\rho_0$ and $\Wrho_0$, respectively.
Let $\pi_0$ be an optimal coupling of $\rho_0$ and $\Wrho_0$.
For $(x, y) \in \BBS^{d-1} \times \BBS^{d-1}$, define
\begin{equation*}
X(t, x) := \PsiRho_t (x), \qquad \WX (t, y) := \PsiWRho_t (y),
\end{equation*}
and set
\begin{equation*}
    D(t) := \int_{\BBS^{d-1} \times \BBS^{d-1}} \left\| X(t,x) - \WX(t, y) \right\| d\pi_0 (x, y) \, .
\end{equation*}
Since $(X(t, \dummy), \WX(t, \dummy))_{\sharp} \pi_0$ is a coupling between $\rho(t)$ and $\Wrho(t)$, we have $W_1 (\rho(t), \Wrho(t)) \leq D(t)$.
Together with Lemma \ref{lem:vf_property}, this implies that
\begin{equation*}
D(t) \leq D(0) + 3C_{\CF} \int_0^t D(s) ds + L_{\CF} \int_0^t W_1 (\rho(s), \Wrho(s))ds \leq D(0) + (3 C_{\CF} + L_{\CF}) \int_0^t D(s) ds \, .
\end{equation*}
Grönwall's inequality then gives
\begin{equation*}
    D(t) \leq e^{(3C_{\CF} + L_{\CF}) t} D(0) \, .
\end{equation*}
Since $\pi_0$ is an optimal coupling, we have $D(0) = W_1 (\rho_0, \Wrho_0)$.
Combining this identity with the preceding estimate gives the Wasserstein stability bound \eqref{eq:ws_stability}.
\end{proof}

\section{Low-dimensional Structures}\label{sec:low-dim_str}
In this section, we explore the low-dimensional structures underlying the multi-particle Oja flow. 
In Section \ref{subsec:NL_transform}, we first study a normalized linear transformation representation of the mean-field dynamics \eqref{eq:mf_oja_flow}.  
Building on this representation,  Section \ref{subsec:inv_param_fam} investigates invariant parametric families of the mean-field equation, within which the evolution of the mean-field distribution can be reduced to the dynamics of the corresponding family parameters.

Throughout the following analysis, $\CF$ satisfies Assumption \ref{asm:well-posedness}, and we assume without loss of generality that 
\begin{equation}\label{eq:traceless_asm}
    \tr{\CF[\varrho]} = 0, \qquad \text{for all } \varrho \in \CP(\BBS^{d-1}) \, .
\end{equation}
Indeed, if this condition does not hold, we may replace $\CF$ by its trace-free part $\CF - \frac{\tr{\CF}}{d} I_d$, leaving the dynamics \eqref{eq:MP_oja_flow} unchanged.
This follows from the shift-invariance property of the associated vector field.
Specifically, since $(I_d - xx^{\top}) x = 0$, for any $c \in \R$, $x \in \BBS^{d-1}$, and $\varrho \in \CP(\BBS^{d-1})$, we have
\begin{equation*}
\left( I_d - xx^{\top} \right) \left( \CF[\varrho] - c I_d  \right) x = \left( I_d - xx^{\top} \right) \CF[\varrho] x \, .
\end{equation*}

\subsection{Representation via Normalized Linear Transformations}\label{subsec:NL_transform}
We begin by introducing the normalized linear transformation \cite{tyler1987statistical,Mahony2003} and collecting several of its basic properties.
We denote the set of invertible matrices in $\R^{d\times d}$ by
\begin{equation*}
    \GL (d, \R) := \left\{ G \in \R^{d\times d}: \det G \neq 0 \right\} \, .
\end{equation*}
For $G \in \GL(d, \R)$, we define the normalized linear (NL) transformation associated with $G$ by
\begin{equation}\label{eq:NL_transform}
 \Phi_G: \BBS^{d-1} \to \BBS^{d-1}, \qquad \Phi_{G}(x) := \frac{Gx}{\|Gx\|} \, .
\end{equation}
This transformation is also referred to as the \emph{unit map} in \cite{lohe2020on,lohe2025exactA,lohe2025exactB}.

Since $G$ is invertible, $G x\neq 0$ for every $x \in \BBS^{d-1}$, and hence $\Phi_G$ is well-defined.
Moreover, for any $G, G' \in \GL(d, \R)$, the NL transformation satisfies the following properties:
\begin{itemize}
    \item \textbf{Bijectivity.} The map $\Phi_G$ is bijective, with inverse $\Phi_{G}^{-1} = \Phi_{G^{-1}}$.
    
    \item \textbf{Scale invariance.} For any nonzero scalar $c \in \R \setminus \{0\}$, $\Phi_{cG} = \sgn {c} \Phi_{G}$.

    \item \textbf{Composition.} $\Phi_{G'} \circ \Phi_{G} = \Phi_{G' G}$.
\end{itemize}
The scale-invariance property shows that every normalized linear transformation can be uniquely represented by a matrix whose determinant has unit absolute value.
Indeed, for any $G \in \GL(d, \R)$, we can always rescale it by the scalar $|\det G|^{-1/d}$, leaving the associated NL transformation unchanged.
Accordingly, we introduce the set
\begin{equation*}
   \SLpm (d, \R) := \left\{ G \in \R^{d\times d}: |\det G| =  1 \right\} \, .
\end{equation*}

We are now ready to establish an exact representation of the mean-field solution to \eqref{eq:mf_oja_flow} in terms of normalized linear transformations.

\begin{theorem}
[Representation via normalized linear transformations]\label{thm:PI_mf_sys}
Suppose that Assumption \ref{asm:well-posedness} and the tracelessness condition \eqref{eq:traceless_asm} hold.
Let $\rho(t)$ be the unique solution to \eqref{eq:mf_oja_flow} with initial condition $\rho(0) = \rho_0 \in \CP(\BBS^{d-1})$.
Then $\rho(t)$ admits the representation
\begin{equation}\label{eq:rho_as_pushforward}
    \rho(t) = (\Phi_{G(t)})_{\sharp} \rho_0,
\end{equation}
where $G: [0,\infty) \to \SLpm(d,\R)$ solves the closed matrix-valued equation
\begin{equation}\label{eq:mf_govern_dyn}
\dot{G}(t) = \CF[(\Phi_{G(t)})_{\sharp} \rho_0] G(t), \qquad G(0) = I_{d} \, .
\end{equation}
\end{theorem}
Theorem \ref{thm:PI_mf_sys} shows that the characteristic flow associated with the mean-field solution $\rho$ of \eqref{eq:mf_oja_flow} is explicitly represented by a normalized linear transformation.
More precisely, the characteristic flow $\Psi_t$ introduced in Theorem \ref{thm:well-posedness_mf} satisfies
\[
\Psi_t = \Phi_{G(t)},
\] 
where $G(t)$ solves a finite-dimensional matrix equation \eqref{eq:mf_govern_dyn}.

The normalized linear transformation representation \eqref{eq:rho_as_pushforward} also provides an explicit description of the representative-particle dynamics discussed in Section \ref{subsec:well-posedness_mf}. 
Indeed, if $X_0$ is a random variable following the distribution $\rho_0$, then 
\begin{equation}
    X(t) = \Phi_{G(t)} (X_0) = \frac{G(t) X_0}{\|G(t) X_0\|}, \qquad t \geq 0
\end{equation}
solves the nonlinear characteristic equation \eqref{eq:representative_particle_ode}.
In other words, the trajectory of a representative particle is obtained by applying the time-dependent normalized linear transformation $\Phi_{G(t)}$ to its initial state.


\begin{remark}[Relation to existing finite-particle results]\label{rem:existing_finite_sys_result}
The normalized linear transformation representation of the finite-particle dynamics was established in \cite{lohe2020on,lohe2025exactA,lohe2025exactB}.
Theorem \ref{thm:PI_mf_sys} recovers the earlier results when specialized to empirical measures.
 
In particular, let $\rho_n(0) = \frac{1}{n} \sum_{k=1}^n \delta_{x_{k, 0}}$ be the initial empirical distribution, and let $\rho_n(t)$ denote the empirical distribution associated with the solution of the finite-particle system \eqref{eq:MP_oja_flow}.
By the consistency between the finite-particle and mean-field descriptions established in Section \ref{subsec:well-posedness_mf}, Theorem \ref{thm:PI_mf_sys} applies to $\rho_n$ and gives
\begin{equation*}
    \rho_n(t) = (\Phi_{G_n(t)})_{\sharp} \rho_n(0) = \frac{1}{n} \sum_{k=1}^n \delta_{\Phi_{G_n(t)} (x_{k, 0})},
\end{equation*}
where $G_n(t)$ satisfies the matrix-valued ODE
\begin{equation}\label{eq:finite_sys_gov_ode}
\dot{G}_n(t) = \CF [(\Phi_{G_n(t)})_{\sharp} \rho_n(0)] G_n(t), \qquad G_n(0) = I_d \, .
\end{equation}
It follows that the unique solution of the finite-particle system \eqref{eq:MP_oja_flow} with initial configuration $x_k(0) = x_{k, 0}$ admits the representation
\begin{equation}\label{eq:PI_finite_sys}
    x_k(t) = \Phi_{G_n(t)} (x_{k, 0}), \qquad k=1,2, \dots, n \, .
\end{equation}
Thus, as shown in \cite{lohe2020on,lohe2025exactA,lohe2025exactB}, the finite-particle dynamics \eqref{eq:MP_oja_flow} are fully determined by the initial configuration and a single ODE whose dimension is independent of the number of particles $n$.
However, it is worth emphasizing that the representation \eqref{eq:PI_finite_sys}, together with equation \eqref{eq:finite_sys_gov_ode}, does not necessarily lead to a reduction in computational cost compared to the original finite-particle system \eqref{eq:MP_oja_flow}.
Indeed, evaluating the right-hand side of \eqref{eq:finite_sys_gov_ode} still requires aggregating information over all particles, so the computational cost generally remains dependent on $n$.
\end{remark}

\begin{remark}[Parallel with the Watanabe--Strogatz theory for the Kuramoto model]\label{rem:connect_to_WS}
The representation of the multi-particle Oja flow presented in Theorem \ref{thm:PI_mf_sys} closely parallels the Watanabe--Strogatz description of the Kuramoto model \cite{watanabe1994constants,marvel2009identical,Lohe_2018,Lipton_2021}.
In that setting, the Kuramoto dynamics are shown to admit a representation through a time-dependent Möbius transformation acting on the initial data, with the transformation governed by finitely many variables.
In the present setting, the analogous representation is given by the normalized linear transformation $\Phi_{G(t)}$, with $G(t)$ evolving according to the matrix equation \eqref{eq:mf_govern_dyn}.

This type of representation is commonly referred to as \emph{partial integrability}.
More specifically, for fixed initial data, the evolution of the system is encoded by a single time-dependent transformation involving only finitely many collective variables \cite{watanabe1993integrability,watanabe1994constants,pikovsky2008partially,Lohe_2018,lohe2025exactA,park2022watanabe}.
\end{remark}

The proof of Theorem \ref{thm:PI_mf_sys} adapts the algebraic calculation from \cite{lohe2020on,lohe2025exactA} to the mean-field setting.
The key observation, which also provides the intuition behind the representation \eqref{eq:rho_as_pushforward}, is that the vector field $b$ of the mean-field equation, defined in \eqref{eq:vf}, coincides with the infinitesimal motion on the sphere generated by the normalized linear transformations; see \eqref{eq:infi_motion} in the proof below.

\begin{proof}[Proof of Theorem \ref{thm:PI_mf_sys}]
Let $\rho$ be the unique global weak solution to \eqref{eq:mf_oja_flow} with initial condition $\rho_0$, and consider the matrix-valued ODE
\begin{equation}\label{eq:Gt_ODE_in_proof}
\dot{G}(t) = \CF \left[ \rho(t) \right] G(t), \qquad G(0) = I_d \, .
\end{equation}
Since $\rho \in C([0, \infty); \CP(\BBS^{d-1}))$ and $\CF$ is Lipschitz continuous under Assumption \ref{asm:well-posedness}, the map $t \mapsto \CF[\rho(t)]$ is continuous.
The equation \eqref{eq:Gt_ODE_in_proof} for $G$ is therefore a linear matrix-valued ODE with continuous coefficients.
The standard existence and uniqueness theory for linear systems \cite{teschl2012ordinary} shows that it admits a unique global solution $G \in C^1 ([0, \infty); \R^{d\times d})$.

We next show that $G(t)$ remains in $\SLpm(d,\R)$.
Suppressing the explicit dependence on $t$ for notational simplicity, Jacobi's formula in adjugate form and \eqref{eq:Gt_ODE_in_proof} give
\begin{equation*}
\frac{d}{dt} \det G = \tr{\adj(G) \dot{G}} = \tr{\adj(G) \CF[\rho] G} = \tr{\CF[\rho] G \, \adj(G)} = \tr{\CF[\rho]} \det G = 0 \, .
\end{equation*}
Here, we used the identity $G\, \adj(G) = (\det G) I_d$ and the tracelessness condition \eqref{eq:traceless_asm}.
Consequently, $\det G(t) = \det G(0) = 1$ for every $t \geq 0$, and therefore $G(t) \in \SLpm(d,\R)$ for all $t \geq 0$.

It remains to establish the representation \eqref{eq:rho_as_pushforward}.
Fix $x \in \BBS^{d-1}$.
Since $G(t)$ is invertible, the normalized linear transformation $\Phi_{G(t)}(x) = \frac{G(t)x}{\|G(t)x\|}$ is well-defined.
Suppressing again the explicit dependence on $t$, differentiating with respect to time gives
\begin{equation}\label{eq:infi_motion}
\begin{aligned}
\frac{d}{dt} \Phi_{G}(x) &= \frac{d}{dt} \frac{Gx}{\|Gx\|} = \frac{\dot{G} x}{\|G x\|} - \frac{G x}{\|G x\|^3} \left\langle \dot{G} x, G x \right\rangle = \frac{\CF[\rho] G x}{\|G x\|} - \frac{G  x}{\|G x\|^3} \left\langle \CF[\rho] G x, G x \right\rangle\\[4pt]
&= \CF[\rho] \Phi_G(x) - \Phi_G(x)\left\langle \CF[\rho] \Phi_G(x), \Phi_G(x) \right\rangle = \left( I_d - \Phi_G(x) \Phi_G(x)^{\top} \right) \CF[\rho] \Phi_G(x) = b(\Phi_G(x), \rho) \, .
\end{aligned}
\end{equation}
Moreover, $\Phi_{G(0)} (x) = \Phi_{I_d} (x) = x$.
Therefore, the family of maps $\Phi_{G(t)}$ satisfies the characteristic equation associated with $\rho$.
By uniqueness of the characteristic flow, $\Phi_{G(t)} = \Psi_t$, where $\Psi_t$ denotes the characteristic flow from Theorem \ref{thm:well-posedness_mf}.
The characteristic representation of the mean-field solution therefore gives
\begin{equation*}
    \rho(t) = (\Psi_t)_{\sharp} \rho_0 = (\Phi_{G(t)})_{\sharp} \rho_0 \, .
\end{equation*}
This proves \eqref{eq:rho_as_pushforward}.
Substituting this representation into \eqref{eq:Gt_ODE_in_proof} gives the closed matrix-valued ODE \eqref{eq:mf_govern_dyn}.
\end{proof}

\subsection{Reduction via Invariant Parametric Families}\label{subsec:inv_param_fam}
Theorem \ref{thm:PI_mf_sys} shows that the solution of the mean-field dynamics \eqref{eq:mf_oja_flow} admits the representation $\rho(t) = (\Phi_{G(t)})_{\sharp} \rho_0$, where $G(t)$ solves the matrix equation \eqref{eq:mf_govern_dyn}.
However, for a general initial distribution $\rho_0 \in \CP(\BBS^{d-1})$, this representation may not provide an explicit description of the solution, since the dependence of $(\Phi_{G(t)})_{\sharp} \rho_0$ on $G(t)$ can be highly intricate.
For the same reason, the analysis and numerical simulation of the matrix equation \eqref{eq:mf_govern_dyn} may also remain challenging, as its right-hand side depends on the evolving pushforward distribution $(\Phi_{G(t)})_{\sharp} \rho_0$.

Motivated by these difficulties, we further explore the structure of the mean-field dynamics \eqref{eq:mf_oja_flow} by identifying \emph{parametric families} of distributions that are \emph{invariant} under the dynamics.
More precisely, we seek families of distributions determined by finitely many parameters with the property that any solution $\rho(t)$ whose initial data belongs to the family remains within the same family for all $t \geq 0$.
We refer to any parametric family with this property as an \emph{invariant parametric family} under the mean-field dynamics \eqref{eq:mf_oja_flow}.

For initial data $\rho_0$ belonging to an invariant parametric family, the evolution of $\rho(t)$ can instead be described directly through the corresponding evolving family parameters, rather than through the pushforward representation $\rho(t) = (\Phi_{G(t)})_{\sharp} \rho_0$.
In this sense, invariant parametric families provide a further reduction in the description of the mean-field dynamics and may thereby facilitate their analytical and numerical study.

In Section \ref{subsubsec:construction_inv_param_fam}, we present a general approach for identifying invariant parametric families on $\BBS^{d-1}$.
Using this approach, we show that projected elliptical distributions \cite{mardia2009directional,stute2005projected,brett1998projected,kato2004further,tsagris2025circular}
and their finite mixtures provide a broad class of examples.
In Section \ref{subsubsec:ACG_manifold_inv}, we examine the important special case of the \emph{angular central Gaussian} family \cite{tyler1987statistical}.
We show that this family is invariant under the mean-field dynamics and derive the reduced evolution equation satisfied by its parameter matrix.

\subsubsection{Identifying Invariant Parametric Families}\label{subsubsec:construction_inv_param_fam}
We begin with a definition that will be useful for identifying the invariant parametric families of interest.
\begin{definition}[Closure under a transformation]
Let $h: \CX \to \CX$ be a measurable transformation, and let $P \subseteq \CP(\CX)$ be a family of probability distributions.
We say that $P$ is \emph{closed under} $h$ if $h_{\sharp} p \in P$ for every $p \in P$.
Equivalently, whenever $X$ is a random variable with $\Law(X) \in P$, one has $\Law(h(X)) \in P$.
\end{definition}
Together with Theorem \ref{thm:PI_mf_sys}, this notion gives a simple sufficient condition for a parametric family to be invariant under the mean-field dynamics \eqref{eq:mf_oja_flow}.
\begin{lemma}\label{lem:closure_implies_inv}
Let $P \subseteq \CP(\BBS^{d-1})$ be a parametric family that is closed under every normalized linear transformation $\Phi_{G}$, with $G \in \GL(d, \R)$.
Then $P$ is invariant under the mean-field equation \eqref{eq:mf_oja_flow}.
More precisely, if $\rho_0 \in P$, the corresponding solution satisfies $\rho(t) \in P$ for every $t \geq 0$.
\end{lemma}

\begin{proof}
By Theorem \ref{thm:PI_mf_sys}, the solution of \eqref{eq:mf_oja_flow} admits the characteristic representation $\rho(t) = (\Phi_{G(t)})_{\sharp} \rho_0$, where $G(t) \in \SLpm(d, \R) \subseteq \GL(d,\R)$ for every $t \geq 0$.
Since $P$ is closed under every normalized linear transformation $\Phi_G$, it follows that $\rho(t) = (\Phi_{G(t)})_{\sharp} \rho_0 \in P$ for every $t \geq 0$, which proves the claim.
\end{proof}
Lemma \ref{lem:closure_implies_inv} reduces the identification of invariant parametric families on $\BBS^{d-1}$ to finding families that are closed under normalized linear transformations.
To find such families, we first introduce some notation. 
For $G \in \GL(d,\R)$, we denote the linear transformation associated with $G$ by
\begin{equation}\label{eq:linear_transformation}
\Lambda_G: \R^d \to \R^d, \qquad \Lambda_G (x) := G x \, .
\end{equation}
We also define the normalization map by
\begin{equation}\label{eq:normalization_map}
  \Pi: \R^d \setminus \{0\} \to \BBS^{d-1}, \qquad  \Pi (x) := \frac{x}{\|x\|} \, .
\end{equation}
The normalized linear transformation \eqref{eq:NL_transform} can then be expressed as $\Phi_G = \Pi \circ \Lambda_G$ on $\BBS^{d-1}$.

Following the standard construction of directional distributions by normalizing Euclidean random vectors \cite{mardia2009directional,pewsey2021recent}, we introduce the following definition.
\begin{definition}[Projected distributions]\label{def:projected_family}
Let $Q \subseteq \CP(\R^d)$ be a family of probability distributions satisfying $q(\{0\}) = 0$ for every $q \in Q$.
The associated family of projected distributions on $\BBS^{d-1}$ is defined by
\begin{equation}\label{eq:projected_family}
    P_Q := \left\{ \Pi_{\sharp} q : q \in Q \right\} \, .
\end{equation}
\end{definition}
By construction, a distribution $p \in P_Q$ can be interpreted as the distribution of the \emph{direction} of a Euclidean random vector whose law belongs to $Q$.
Here, the condition $q(\{0\}) = 0$ ensures that the normalization map $\Pi$ is defined $q$-almost surely for every $q \in Q$.

The relationship between Euclidean random vectors and their directions leads to a useful observation: closure under linear transformations in $\R^d$ is inherited, after projection, as  closure under normalized linear transformations on $\BBS^{d-1}$.
Combining this observation with Lemma \ref{lem:closure_implies_inv}, we obtain the following sufficient condition for finding the invariant parametric families of interest.

\begin{proposition}[Identification of invariant parametric families]\label{prop:construct_inv}
Let $Q \subseteq \CP(\R^d)$ be a parametric family satisfying $q(\{0\}) = 0$ for every $q \in Q$.
Suppose that $Q$ is closed under every linear transformation $\Lambda_G$, with $G \in \GL(d, \R)$.
Then the associated projected family $P_Q$, defined in \eqref{eq:projected_family}, is invariant under the mean-field dynamics \eqref{eq:mf_oja_flow}.
\end{proposition}

\begin{proof}
By Lemma \ref{lem:closure_implies_inv}, it suffices to show that $P_Q$ is closed under every normalized linear transformation $\Phi_G$.

Let $p = \Pi_{\sharp} q \in P_Q$ for some $q \in Q$, and fix $G \in \GL(d, \R)$.
Using the composition rule for pushforward measures and the identity $\Phi_G \circ \Pi = \Pi \circ \Lambda_G$ on $\R^d \setminus \{0\}$, we obtain
\begin{equation*}
(\Phi_G)_{\sharp} p = (\Phi_G)_{\sharp} (\Pi_{\sharp} q) = (\Phi_G \circ \Pi)_{\sharp} q = (\Pi \circ \Lambda_G)_{\sharp} q = \Pi_{\sharp} ((\Lambda_G)_{\sharp} q) \, .
\end{equation*}
Since $Q$ is closed under $\Lambda_G$, we have $(\Lambda_G)_{\sharp} q \in Q$.
It follows that $(\Phi_G)_{\sharp} p \in P_Q$.
This shows that $P_Q$ is closed under normalized linear transformations.
\end{proof}

Proposition \ref{prop:construct_inv} provides a simple yet effective procedure for finding invariant parametric families on $\BBS^{d-1}$: one first identifies a parametric family $Q$ on $\R^d$ that is closed under linear transformations and then considers its associated projected family $P_Q$, as defined in \eqref{eq:projected_family}.

\paragraph{Parametric families closed under linear transformations.}
Many important parametric families of distributions on $\R^d$ are closed under linear transformations.
A broad class of examples is provided by elliptical distributions \cite{cambanis1981on,fang1990symmetric}, which include, among others, multivariate Gaussian distributions, multivariate $t$-distributions, and symmetric multivariate Laplace distributions.
We refer to \cite{cambanis1981on,fang1990symmetric} for a comprehensive treatment of elliptical distributions and their transformation properties.

\begin{remark}
The closure property under linear transformations extends naturally to finite mixtures of elliptical distributions.
As a simple illustration, consider a two-component Gaussian mixture 
\begin{equation*}
    Y \sim \alpha \CN (\mu_1, \Sigma_1) + (1-\alpha) \CN (\mu_2, \Sigma_2), \qquad \alpha \in [0,1] \, .
\end{equation*}
Then, for any $G \in \GL(d, \R)$,
\begin{equation*}
   \Lambda_G(Y) \sim \alpha \CN (G\mu_1, G\Sigma_1 G^{\top}) + (1-\alpha) \CN (G\mu_2, G\Sigma_2 G^{\top}) \, .
\end{equation*}
Thus, the family of two-component Gaussian mixtures is closed under linear transformations, with both the number of mixture components and their associated weights preserved.
More generally, the same observation applies to finite mixtures of any family of distributions that is itself closed under linear transformations.
\end{remark}

\paragraph{Invariant parametric families.}
By Proposition \ref{prop:construct_inv}, projecting the Euclidean parametric families identified above onto the unit sphere gives a broad class of parametric families on $\BBS^{d-1}$ that are invariant under the mean-field dynamics \eqref{eq:mf_oja_flow}.
An important example is the class of projected elliptical distributions \cite{mardia2009directional,stute2005projected}, which includes projected Gaussian distributions \cite{pukkila1988pattern,brett1998projected,HernandezStumpfhauser2017TheGP}, projected $t$-distributions \cite{kato2004further}, and projected Cauchy distributions \cite{tsagris2025circular,alzeley2026generalized}.
This construction also extends to projected families arising from finite mixtures of elliptical distributions.

Identifying such invariant parametric families provides additional structural information about the mean-field equation \eqref{eq:mf_oja_flow}.
In particular, it shows that any solution $\rho(t)$ initialized within one of these families remains in the same prescribed class for all time, thereby restricting the set of distributions that may arise along the evolution.
Moreover, the mean-field dynamics can then be studied through the evolution of the corresponding family parameters.

In the next subsection, we illustrate this reduction using one such invariant family, namely the angular central Gaussian family \cite{tyler1987statistical}.


\subsubsection{Reduction within the Angular Central Gaussian Family}\label{subsubsec:ACG_manifold_inv}
We begin by recalling the definition of the angular central Gaussian distribution \cite{tyler1987statistical}.

Following the standard construction of directional distributions as in Definition \ref{def:projected_family}, the angular central Gaussian family is obtained by projecting centered Gaussian distributions on $\R^d$ onto the unit sphere $\BBS^{d-1}$.
More precisely, let $\CN(0,\Sigma)$ denote the centered Gaussian distribution with covariance matrix $\Sigma$, where $\Sigma \in \R^{d\times d}$ is symmetric and positive definite. 
If
\begin{equation*}
    Y \sim \CN(0,\Sigma) \qquad \text{and} \qquad X = \Pi(Y) = \frac{Y}{\|Y\|},
\end{equation*}
then $X$ is said to follow the angular central Gaussian distribution with parameter $\Sigma$, denoted by $X \sim \ACG(\Sigma)$.
Equivalently,
\begin{equation*}
    \ACG(\Sigma) = \Pi_{\sharp} \CN (0, \Sigma) \, .
\end{equation*}
The angular central Gaussian distribution $\ACG(\Sigma)$ admits a density with respect to the surface measure on $\BBS^{d-1}$ given by
\begin{equation}\label{eq:ACG_density}
    p_{\Sigma} (x) := \frac{\Gamma(d/2)}{2\pi^{d/2}} \frac{1}{\sqrt{\det \Sigma}} (x^{\top} \Sigma^{-1} x)^{-d/2}, \qquad x \in \BBS^{d-1} \, .
\end{equation}
With a slight abuse of notation, we henceforth identify $\ACG(\Sigma)$ with its density $p_{\Sigma}$ and use the two notations interchangeably to denote the angular central Gaussian distribution.

The parameter $\Sigma$ is identifiable only up to multiplication by a positive scalar.
Indeed, for every $c > 0$, $p_{c\Sigma} (x) = p_{\Sigma} (x)$ for every $x \in \BBS^{d-1}$.
Therefore, we remove this ambiguity by restricting the parameter space of $\ACG(\Sigma)$ to
\begin{equation*}
    \SPD_1 (d) := \left\{ \Sigma \in \R^{d\times d}: \Sigma = \Sigma^{\top}, \Sigma \succ 0, \det \Sigma = 1 \right\} \, .
\end{equation*}
A visual illustration of the angular central Gaussian family is provided in Figure \ref{fig:ACG}.
We refer to \cite{tyler1987statistical} for a more detailed treatment of this family and its basic properties.

We are now ready to show that the angular central Gaussian family is invariant under the mean-field equation \eqref{eq:mf_oja_flow}.

\begin{figure}[!htb]
    \centering
    \includegraphics[width=0.85\linewidth]{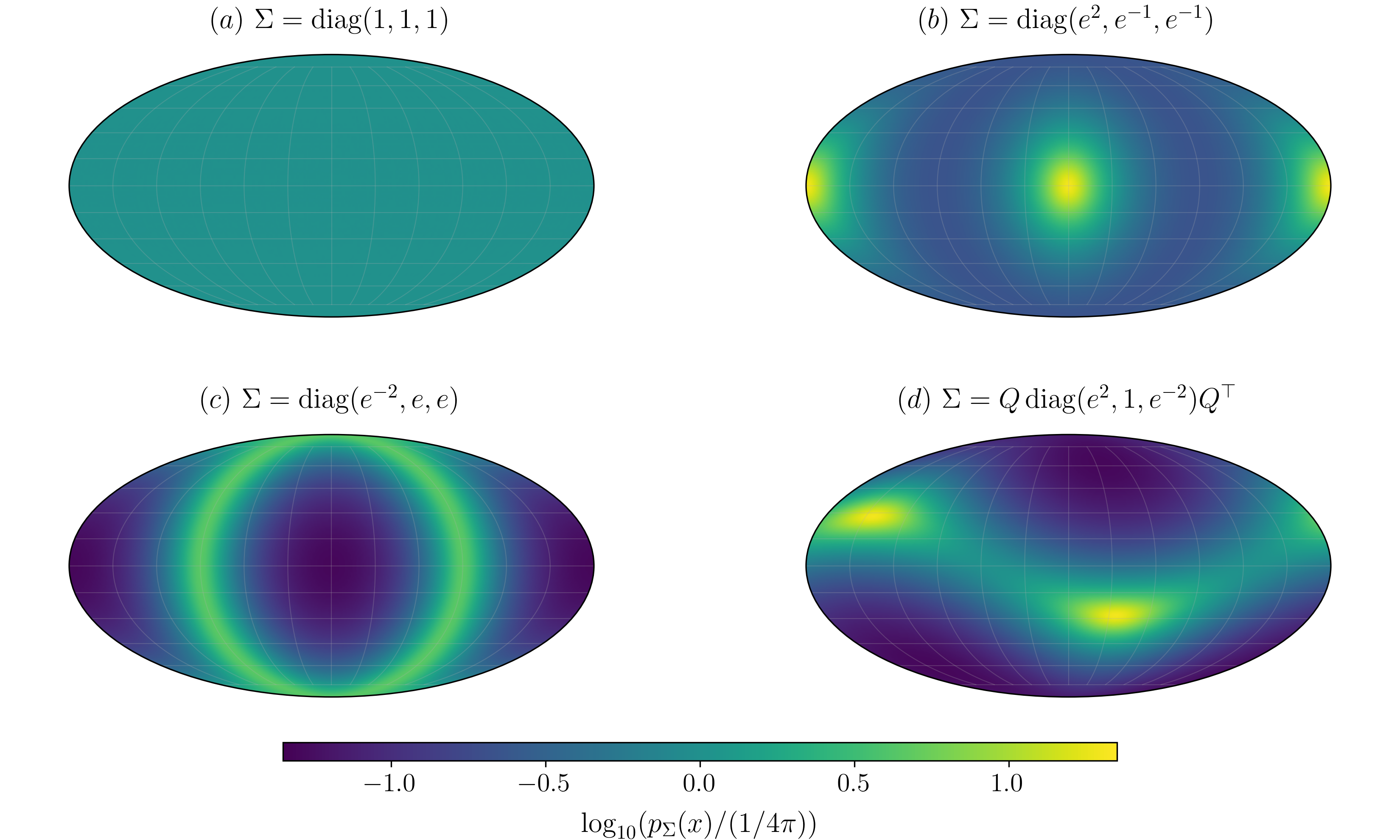}
    \caption{Mollweide-projection visualizations of angular central Gaussian densities on the unit sphere $\BBS^2$ for different choices of the parameter matrix $\Sigma$. 
    In panel (d), $Q$ denotes a rotation matrix.}
    \label{fig:ACG}
\end{figure}

\begin{theorem}[Reduction within the angular central Gaussian family]\label{thm:ACG_invariant}
Let $\rho(t)$ be the solution of the mean-field dynamics \eqref{eq:mf_oja_flow} with initial condition $\rho(0) = \rho_0$, and suppose that the interaction kernel $\CF$ satisfies Assumption \ref{asm:well-posedness} and the tracelessness condition \eqref{eq:traceless_asm}.
Assume that $\rho_0$ belongs to the angular central Gaussian family, i.e., $\rho_0 = \ACG(\Sigma_0)$ for some $\Sigma_0 \in \SPD_1(d)$. 
Then, for all $t \geq 0$,
\begin{equation}
    \rho(t) = \ACG (\Sigma(t)),
\end{equation}
where $\Sigma: [0, \infty) \to \SPD_1(d)$ solves
\begin{equation}\label{eq:mf_goven_dyn_acg}
    \dot{\Sigma}(t) = \CF[p_{\Sigma(t)}] \Sigma(t) + \Sigma(t) \CF[p_{\Sigma(t)}]^{\top}, \qquad \Sigma(0) = \Sigma_0 \, .
\end{equation}
\end{theorem}
Theorem \ref{thm:ACG_invariant} shows that, for angular central Gaussian initial data, the mean-field distribution $\rho(t)$ remains within the same family for all $t\geq 0$.
Moreover, within this family, the general representation \eqref{eq:rho_as_pushforward} established in Theorem \ref{thm:PI_mf_sys} takes the reduced form $\rho(t) = p_{\Sigma(t)}$, where the density $p_{\Sigma(t)}$ is given by \eqref{eq:ACG_density} and the parameter $\Sigma(t)$ evolves according to \eqref{eq:mf_goven_dyn_acg}.
We refer to this finite-dimensional representation of the mean-field dynamics within the angular central Gaussian family as the \emph{ACG reduction}.

The ACG reduction allows the mean-field evolution \eqref{eq:mf_oja_flow} to be studied through the finite-dimensional dynamics of $\Sigma(t)$, which can be simpler to analyze than the evolution on the space of probability measures.

A concrete illustration of this simplification is provided by the characterization of stationary states in the following proposition, whose proof is deferred to Appendix \ref{app:ACG}.

\begin{proposition}\label{prop:stationary_states}
Suppose that the assumptions of Theorem \ref{thm:ACG_invariant} hold and, in addition, that the interaction kernel $\CF$ is symmetric. 
That is, $\CF[\varrho] = \CF[\varrho]^{\top}$ for every $\varrho \in \CP(\BBS^{d-1})$. 
Then the set of stationary distributions of the original mean-field equation \eqref{eq:mf_oja_flow} is 
\begin{equation}\label{eq:stat_state_mf}
\CS := \left\{ \rho_* \in \CP(\BBS^{d-1}) : (I_d - x x^{\top}) \CF[\rho_*] x = 0, \;\; \rho_*\text{-a.e.} \right\} \, .
\end{equation}
Within the angular central Gaussian family, the set of stationary distributions is
\begin{equation}\label{eq:stat_state_ACG}
\CS_{\ACG} := \left\{ p_{\Sigma_*} : \CF[p_{\Sigma_*}] = 0, \, \Sigma_* \in \SPD_1(d) \right\} \, .
\end{equation}
In particular, $\CS_{\ACG} \subseteq \CS$.
\end{proposition}

Proposition \ref{prop:stationary_states} provides a simple illustration of how the ACG reduction can simplify the analysis of stationary states.
In particular, the condition $(I_d - xx^{\top}) \CF[\rho_*] x = 0$, $\rho_*\text{-a.e.}$, does not in general imply $\CF[\rho_*] = 0$, since $\rho_*$ may be supported entirely on eigenvectors of a nonzero matrix $\CF[\rho_*]$.  
Within the angular central Gaussian family, however, stationary distributions are fully characterized by a simpler finite-dimensional condition: $\CF[p_{\Sigma_*}] = 0$ with $\Sigma_* \in \SPD_1(d)$.
We revisit this comparison of stationary states in Section \ref{sec:examples} for two representative interaction kernels listed in Table \ref{tab:interaction-kernels}, where the ACG reduction leads to more explicit characterizations of the corresponding stationary states.

The ACG reduction can also facilitate the study of the long-time behavior of the system, particularly when the interaction kernel $\CF$ possesses additional structure.
For example, when $\CF$ depends on the distribution through quantities such as its moments, $\CF[p_{\Sigma(t)}]$, where $p_{\Sigma(t)}$ is the ACG density given in \eqref{eq:ACG_density}, may admit a tractable expression.
Such an expression can make the resulting finite-dimensional system for $\Sigma(t)$ more amenable to both analytical and numerical study than the original mean-field dynamics.
Section \ref{sec:examples} provides concrete examples of such reduced systems, which offer a tractable setting for further investigation of their long-time behavior.
By contrast, the long-time behavior of the corresponding mean-field dynamics may be more difficult to characterize directly.


The invariant parametric family viewpoint underlying the ACG reduction is in the same spirit as the Ott--Antonsen ansatz for the Kuramoto model \cite{kuramoto1975international,kuramoto1984chemical}, where the dynamics likewise preserve a parametric family of probability distributions.
We discuss this connection in more detail in the following remark.

\begin{remark}[Parallel with the Ott--Antonsen ansatz for the Kuramoto model]\label{rem:connect_to_OA}
The invariance of the angular central Gaussian family under the multi-particle Oja flow is analogous to the Ott--Antonsen ansatz for the Kuramoto model \cite{ott2008low,ott2009long}.
In particular, in dimension $d=2$, the wrapped Cauchy family plays the analogous role for the classical Kuramoto model \cite{ott2008low,ott2009long,marvel2009identical}, forming the well-known Ott--Antonsen invariant manifold.
For the higher-dimensional Kuramoto model on the unit sphere, the associated invariant family is the spherical Cauchy family \cite{chandra2019complexity,Lipton_2021,kato2020some}.

A further parallel arises from the transformation structure underlying these families. 
The spherical Cauchy family can be obtained by pushing forward the uniform distribution on the unit sphere through Möbius transformations \cite{chandra2019complexity,Lipton_2021,kato2020some}.
Analogously, the angular central Gaussian family can be obtained by pushing forward the uniform distribution through normalized linear transformations, as shown in Proposition \ref{prop:M_Unif_ACG_manifold}.

The present setting, however, reveals a broader structural picture.
In Section \ref{subsubsec:construction_inv_param_fam}, we develop a systematic approach for identifying invariant parametric families for the multi-particle Oja flow by exploiting closure under normalized linear transformations.
This leads to a broad class of invariant parametric families, of which the angular central Gaussian family is one example.

We note that, on the circle, other parametric families closed under Möbius transformations have also been identified, such as the Kato--Jones family \cite{kato2010family,jacimovic2025learning}. 
These examples, together with the parallel discussed above, naturally raise the question of whether there is a systematic approach for identifying parametric families that are closed under Möbius transformations and hence invariant under the corresponding Kuramoto dynamics.
\end{remark}

We conclude this section by presenting the proof of Theorem \ref{thm:ACG_invariant}.

\begin{proof}[Proof of Theorem \ref{thm:ACG_invariant}]
Suppose that $\rho_0 = \ACG(\Sigma_0)$ for some $\Sigma_0 \in \SPD_1(d)$.
By Theorem \ref{thm:PI_mf_sys}, the mean-field distribution admits the characteristic representation
\begin{equation*}
    \rho(t) = (\Phi_{G(t)})_{\sharp} \ACG(\Sigma_0),
\end{equation*}
where $G: [0, \infty) \to \SLpm(d,\R)$ satisfies
\begin{equation*}
    \Dot{G}(t) = \CF[(\Phi_{G(t)})_{\sharp} \ACG(\Sigma_0)] G(t), \qquad G(0) = I_d \, .
\end{equation*}

Using the composition rule for pushforward measures and the identity $\Phi_G \circ \Pi = \Pi \circ \Lambda_G$ on $\R^d\setminus \{0\}$ for any $G \in \GL(d,\R)$, we can compute that
\begin{equation*}
\begin{aligned}
\rho(t) &= (\Phi_{G(t)})_{\sharp} \ACG(\Sigma_0) = (\Phi_{G(t)})_{\sharp} \left( \Pi_{\sharp} \CN (0, \Sigma_0) \right) = \left( \Phi_{G(t)} \circ \Pi \right)_{\sharp} \CN (0, \Sigma_0)\\[4pt]
&= \left( \Pi \circ \Lambda_{G(t)} \right)_{\sharp} \CN (0, \Sigma_0) = \Pi_{\sharp} \CN (0, G(t)\Sigma_0 G(t)^{\top}) = \ACG(\Sigma(t)),
\end{aligned}
\end{equation*}
where we denote 
\[
\Sigma(t) := G(t) \Sigma_0 G(t)^{\top} \, .
\]
Since $G(t) \in \SLpm(d, \R)$ and $\Sigma_0 \in \SPD_1(d)$, the matrix $\Sigma(t)$ is symmetric positive definite and satisfies $\det \Sigma(t) = \det G(t) \det \Sigma_0 \det G(t)^{\top} = 1$, and thus $\Sigma(t) \in \SPD_1(d)$ for all $t \geq 0$.
Finally, differentiating $\Sigma(t)$ and using the equation for $G(t)$, we obtain
\begin{equation*}
\begin{aligned}
\Dot{\Sigma}(t) &= \Dot{G}(t) \Sigma_0 G(t)^{\top} +  G(t) \Sigma_0 \Dot{G}(t)^{\top} = \CF[\ACG(\Sigma(t))] G(t) \Sigma_0 G(t)^{\top} + G(t)\Sigma_0 G(t)^{\top} \CF[\ACG(\Sigma(t))]^{\top}\\[4pt]
& = \CF[p_{\Sigma(t)}] \Sigma(t) + \Sigma(t) \CF[p_{\Sigma(t)}]^{\top} \, .
\end{aligned}
\end{equation*}
The initial condition is $\Sigma(0) = I_d \Sigma_0 I_d^{\top} = \Sigma_0$.
This completes the proof.
\end{proof}

\section{Example Systems}\label{sec:examples}
In this section, we illustrate the low-dimensional descriptions of the multi-particle Oja flow developed in Section \ref{sec:low-dim_str} through several example systems.
We begin with a distribution-independent interaction kernel $\CF$ as a basic illustrative example, and then consider two selected distribution-dependent kernels listed in Table \ref{tab:interaction-kernels}.

\paragraph{Example 1.}
As a consistency check, we first consider the distribution-independent interaction kernel 
\begin{equation}
    \CF[\varrho] = B, \qquad \varrho \in \CP(\BBS^{d-1}),
\end{equation}
where $B \in \R^{d\times d}$ is fixed and satisfies $\tr{B} = 0$.
In this case, the multi-particle Oja flow  \eqref{eq:MP_oja_flow} reduces to $n$ independent copies of the classical Oja flow \cite{oja1982simplified}.
It therefore suffices to consider
\begin{equation*}
    \dot{X} = \left( I_d - X X^{\top} \right) B X \, .
\end{equation*}
We first note that Theorem \ref{thm:PI_mf_sys} recovers the explicit solution of the classical Oja flow.
Indeed, the solution admits the representation
\begin{equation*}
    X(t) = \frac{G(t)X(0)}{\|G(t) X(0)\|},
\end{equation*}
where $G(t)$ solves \eqref{eq:mf_govern_dyn}, which in this case reduces to 
\begin{equation*}
    \dot{G} = B G, \qquad G(0) = I_d \, .
\end{equation*}
Solving the ODE gives $G(t) = e^{tB}$, and then we obtain the explicit formula established in \cite{yan1994global},
\begin{equation*}
    X(t) = \frac{e^{tB} X(0)}{\|e^{tB} X(0)\|} \, .
\end{equation*}
When $B$ is symmetric and has a unique largest eigenvalue, this representation gives the convergence behavior of the classical Oja flow.
Let $v_1$ be a unit eigenvector associated with the largest eigenvalue of $B$.
Then, for any initial condition satisfying $X(0)^{\top} v_1 \neq  0$, we have
\[
X(t) \rightarrow \sgn {X(0)^{\top} v_1} v_1, \qquad \text{as } t \rightarrow \infty \, .
\]
We refer to  \cite{oja1982simplified,oja1985stochastic,yan1994global} for the corresponding convergence results.

Beyond this trajectory-wise characterization, the ACG reduction established in Theorem \ref{thm:ACG_invariant} provides a finite-dimensional description of the corresponding distribution.
In particular, if $X(0) \sim \rho_0 = \ACG(\Sigma_0)$ for some $\Sigma_0 \in \SPD_1(d)$, then
\[
X(t) \sim \rho(t) = \ACG(\Sigma(t))
\] 
for all $t \geq 0$, where $\Sigma(t)$ solves 
\begin{equation}
    \dot{\Sigma}(t) = B \Sigma(t) + \Sigma(t) B^{\top}, \qquad \Sigma(0) = \Sigma_0 \, .
\end{equation}
This equation admits the explicit solution
\begin{equation}\label{eq:ACG_Sigma_classical_Oja_flow}
    \Sigma(t) = e^{tB} \Sigma_0 e^{tB^{\top}} \, .
\end{equation}
Under the same assumption that $B$ is symmetric with a unique largest eigenvalue and that $\Sigma_0$ is positive definite, we further have
\begin{equation*}
 \frac{\Sigma(t)}{\tr{\Sigma(t)}} \rightarrow v_1 v_1^{\top}, \qquad \text{as } t \rightarrow \infty;
\end{equation*}
see Lemma \ref{lem:asm_behavior_Sigma_t} for the derivation.
Since $\tr{\Sigma(t)} > 0$ for all $t\geq 0$, the positive scale invariance of the angular central Gaussian distribution gives
\begin{equation*}
    \rho(t) = \ACG(\Sigma(t)) = \ACG \left(\frac{\Sigma(t)}{\tr{\Sigma(t)}}\right) \, .
\end{equation*}
Consequently,
\begin{equation*}
 \rho(t) \rightharpoonup \frac{1}{2} \delta_{v_1} + \frac{1}{2} \delta_{-v_1}, \qquad \text{as } t\rightarrow \infty \, .
\end{equation*}
Thus, the ACG reduction provides a distribution-level counterpart to the classical trajectory-wise convergence result: for angular central Gaussian initial data, the evolving distribution asymptotically concentrates on the antipodal pair $\pm v_1$.

Although these observations are not new, they provide a useful consistency check on the results established in Section \ref{sec:low-dim_str}.
We next turn to some distribution-dependent interaction kernels listed in Table \ref{tab:interaction-kernels}.
For these kernels, the equation \eqref{eq:mf_govern_dyn} for $G(t)$ remains coupled to the evolving distribution and, in general, no longer admits an explicit solution of the type available in the present example.
In this setting, the reductions via invariant parametric families developed in Section \ref{subsec:inv_param_fam} can provide a more concrete description of the mean-field evolution.

\paragraph{Example 2.} 
We consider the second-moment interaction kernel \cite{liu2005axial,lohe2020on,kim2021cluster,zhang2022opinion,lohe2025exactA} in the following traceless form:\footnote{The term $\frac{1}{d}I_d$ is subtracted only to ensure that $\CF[\varrho]$ is traceless; this does not change the associated dynamics.}
\begin{equation}\label{eq:CF_second_moment}
    \CF[\varrho] = \int_{\BBS^{d-1}} x x^{\top} d\varrho(x) - \frac{1}{d} I_d, \qquad  \varrho \in \CP(\BBS^{d-1}) \, . 
\end{equation}
In this case, the multi-particle Oja flow \eqref{eq:MP_oja_flow} takes the explicit form
\begin{equation}\label{eq:MP_oja_flow_second_moment_kernel}
    \dot{x}_k = \left( I_d - x_k x_k^{\top} \right) \frac{1}{n} \sum_{j=1}^n (x_j x_j^{\top}) x_k, \qquad k = 1,2, \dots, n \, .
\end{equation}
At the mean-field level, let $\rho(t)$ denote the solution of the continuity equation \eqref{eq:mf_oja_flow} associated with the interaction kernel $\CF$ defined in \eqref{eq:CF_second_moment}.

Applying the ACG reduction established in Theorem \ref{thm:ACG_invariant}, if the initial data $\rho_0 = \ACG(\Sigma_0)$ for some $\Sigma_0 \in \SPD_1(d)$, 
then
\[
\rho(t) = \ACG(\Sigma(t))
\]
for all $t \geq 0$, where $\Sigma(t)$ solves 
\begin{equation}\label{eq:ACG_param_dyn_second_moment}
\dot{\Sigma}(t) = M(\Sigma(t)) \Sigma(t) + \Sigma(t)  M(\Sigma(t)) - \frac{2}{d} \Sigma(t), \qquad \Sigma(0) = \Sigma_0 \, .
\end{equation}
Here, for $\Sigma \in \SPD_1(d)$, we denote by $M(\Sigma)$ the second moment matrix of the angular central Gaussian distribution $\ACG(\Sigma)$:
\begin{equation}\label{eq:ACG_second_moment}
M(\Sigma) := \int_{\BBS^{d-1}} x x^{\top} dp_{\Sigma}(x)  \, .
\end{equation}
Accordingly, $\CF[p_{\Sigma}] = M(\Sigma) - \frac{1}{d} I_d$.
Moreover, $M(\Sigma)$ admits the one-dimensional integral representation \cite{ospald2021modeling}:
\begin{equation}\label{eq:ACG_second_moment_main}
    M(\Sigma) = \frac{1}{2} \int_0^{\infty} \frac{\Sigma (I_d + s \Sigma)^{-1}}{\sqrt{\det (I_d+s\Sigma)}} ds \, .
\end{equation}

The reduced equation \eqref{eq:ACG_param_dyn_second_moment} provides a finite-dimensional description that fully determines the mean-field distribution through $\rho(t) = \ACG(\Sigma(t))$, and can thereby simplify both analytical and computational studies of the mean-field dynamics.

As a first illustration, we compare the stationary distributions within the ACG family with those of the original mean-field equation.
For the second-moment interaction kernel $\CF$ defined in \eqref{eq:CF_second_moment}, Proposition \ref{prop:stationary_states} shows that the stationary distributions within the ACG family are given by
\begin{equation*}
    \CS_{\ACG} = \left\{ \ACG(\Sigma_*) : M(\Sigma_*) = \frac{1}{d} I_d, \;\; \Sigma_* \in \SPD_1(d) \right\} \, .
\end{equation*}
Using the integral representation of $M(\Sigma)$ given in \eqref{eq:ACG_second_moment_main}, we show in Lemma \ref{lem:stat_dis_example_2} that
\begin{equation*}
\CS_{\ACG} = \left\{ \ACG (I_d) \right\} \, .
\end{equation*}
Namely, the uniform distribution on $\BBS^{d-1}$ is the unique stationary distribution within the ACG family.

The original mean-field equation \eqref{eq:mf_oja_flow}, however, admits additional stationary distributions.
Indeed, Proposition \ref{prop:stationary_states} shows that a probability distribution $\rho_* \in \CP(\BBS^{d-1})$ is stationary for \eqref{eq:mf_oja_flow} with the second-moment interaction kernel $\CF$ if and only if
\begin{equation*}
    (I_d - x x^{\top}) \CF[\rho_*] x = 0, \qquad \rho_*\text{-a.e.}
\end{equation*}
For example, for any $v \in \BBS^{d-1}$, the distribution $\rho_* = \frac{1}{2} \left( \delta_{v} + \delta_{-v} \right)$ satisfies the above condition and is therefore stationary.
To see this, for $x = \pm v$, \eqref{eq:CF_second_moment} gives
\begin{equation*}
    \left( I_d - x x^{\top} \right) \CF[\rho_*] x = \left( I_d - x x^{\top} \right) \left( 1 - \frac{1}{d} \right)x = 0 \, .
\end{equation*}
Therefore, the uniform distribution is the unique stationary distribution within the ACG family, whereas the original mean-field equation possesses other stationary distributions outside this family. 
This illustrates how restricting the dynamics to the ACG family narrows the stationary-state problem to a smaller class that can be characterized directly through the finite-dimensional parameter $\Sigma$.

The finite-dimensional dynamics \eqref{eq:ACG_param_dyn_second_moment} of $\Sigma(t)$ also provide a direct way to numerically simulate the mean-field evolution.
Rather than simulating the interacting particle system \eqref{eq:MP_oja_flow_second_moment_kernel} with a large number of particles and approximating $\rho(t)$ by the corresponding empirical distribution, one can directly solve \eqref{eq:ACG_param_dyn_second_moment} without introducing a particle approximation of the mean-field distribution.

In Figure \ref{fig:second_moment_sigma}, we illustrate this approach in dimension $d=3$ by solving \eqref{eq:ACG_param_dyn_second_moment} and sampling from the resulting distributions $\ACG(\Sigma(t))$ at several time points.
We compare these distributions with the empirical distributions obtained by directly simulating the particle system \eqref{eq:MP_oja_flow_second_moment_kernel}, shown in Figure \ref{fig:oja-second-moment}.
The close agreement between the two simulations illustrates the consistency between the large-$n$ finite-particle evolution and the mean-field evolution described by the ACG reduction. 

\begin{figure}[!htb]
    \centering
    \includegraphics[width=1.0\linewidth]{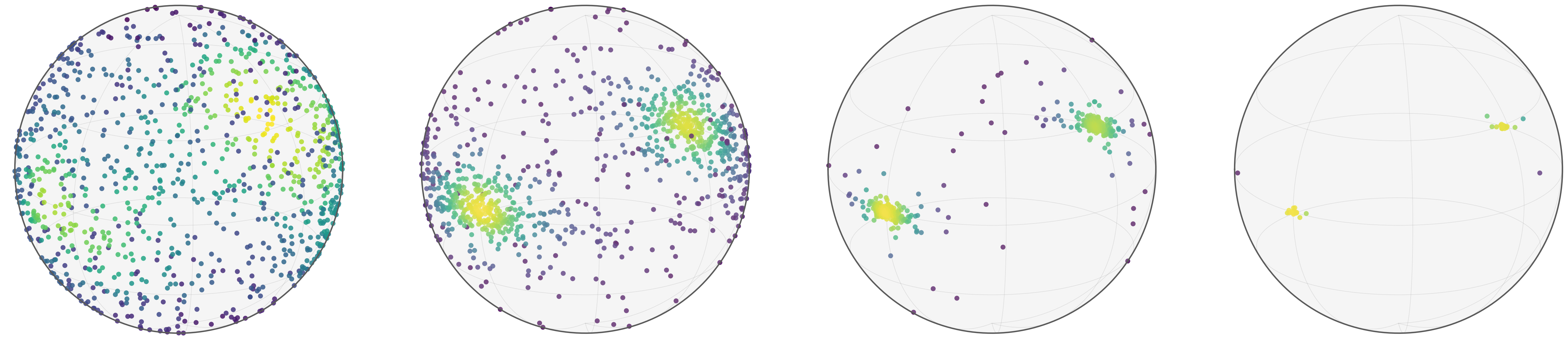}
    \caption{Evolution of the mean-field distribution $\rho(t) = \ACG(\Sigma(t))$ in dimension $d=3$, where $\Sigma(t)$ solves \eqref{eq:ACG_param_dyn_second_moment}.
    The four snapshots visualize the evolving distribution on $\BBS^2$ by displaying samples drawn from $\ACG(\Sigma(t))$ at successive time points.
    }
    \label{fig:second_moment_sigma}
\end{figure}

The reduced dynamics \eqref{eq:ACG_param_dyn_second_moment} also provide a natural setting for investigating long-time behavior of the mean-field system within the ACG family.
In the following remark, we briefly review existing results on the long-time behavior of the finite-particle system \eqref{eq:MP_oja_flow_second_moment_kernel} and discuss how the ACG reduction provides a framework for analysis at the mean-field level.

\begin{remark}[Analysis of long-time behavior]\label{rem:long-time_behavior_example_2}
The long-time behavior of the finite-particle system \eqref{eq:MP_oja_flow_second_moment_kernel} has been studied theoretically in previous works \cite{kim2021cluster,zhang2022opinion}.
In particular, under suitable assumptions imposing a \emph{two-cluster} structure on the initial particle positions, the system has been shown to converge asymptotically to bipartite consensus \cite{kim2021cluster}. 
That is, each particle converges to one of a pair of antipodal points.
The ACG reduction developed in the present paper provides a complementary approach for investigating this collective behavior at the mean-field level.
In particular, the reduced equation \eqref{eq:ACG_param_dyn_second_moment} provides a direct route to studying the long-time behavior of the mean-field dynamics for angular central Gaussian initial data.

It has also been shown that, under suitable assumptions imposing a \emph{one-cluster} structure on the initial particle configuration, the finite-particle system \eqref{eq:MP_oja_flow_second_moment_kernel} converges asymptotically to complete consensus \cite{kim2021cluster}.
That is, all particles concentrate on a single point.
Such behavior cannot be captured within the angular central Gaussian family, since every distribution in this family is antipodally symmetric.
This limitation, however, is specific to the choice of invariant family rather than to the reduction framework developed in Section \ref{subsec:inv_param_fam}. 
In particular, one may instead consider asymmetric invariant parametric families of the multi-particle Oja flow, such as projected Gaussian distributions with nonzero mean \cite{pukkila1988pattern,brett1998projected,HernandezStumpfhauser2017TheGP}.

A detailed investigation of the reduced dynamics associated with different invariant parametric families, including their long-time behavior, is left for future work.
\end{remark}

\begin{remark}[Special case: dimension $d=2$]
In dimension $d=2$, writing $x_k = (\cos \theta_k, \sin \theta_k) \in \BBS^1$, the system \eqref{eq:MP_oja_flow_second_moment_kernel} becomes
\begin{equation*}
    \dot{\theta}_k = \frac{1}{2n} \sum_{j=1}^n \sin (2 (\theta_j - \theta_k)), \qquad k=1,\dots, n \, .
\end{equation*}
That is, the system coincides with the Kuramoto--Daido model with pure second-harmonic coupling \cite{daido1992order,generic1994daido}.
A standard approach to studying this model is to introduce the double-angle transformation $\phi_k = 2\theta_k$ \cite{skardal2011cluster,chen2019low}, under which the system becomes the classical first-harmonic Kuramoto model \cite{kuramoto1975international,kuramoto1984chemical}.
In this way, the Watanabe--Strogatz transformation \cite{watanabe1994constants,marvel2009identical} and the Ott--Antonsen ansatz \cite{ott2008low,ott2009long} can be applied.

This correspondence also connects the ACG reduction to the Ott--Antonsen ansatz in this special case.
Under the same double-angle map, the angular central Gaussian family on $\BBS^1$ corresponds to the wrapped Cauchy family \cite{Kent1988maximum,ciobotaru2018mean}, which is the invariant family associated with the Ott--Antonsen ansatz for the classical Kuramoto model.
Thus, in dimension $d=2$, the ACG reduction for the model \eqref{eq:MP_oja_flow_second_moment_kernel} corresponds, under the double-angle transformation, to the Ott--Antonsen reduction for the classical Kuramoto model.
\end{remark}

\paragraph{Example 3.}
We next consider the cross-product interaction kernel in dimension $d=3$ \cite{lohe2025exactA}, written in the following traceless form:
\begin{equation}\label{eq:CF_three_body}
\CF[\varrho] = -\int_{\BBS^2} \int_{\BBS^2} \left( (x \times y) (x \times y)^{\top} - \frac{1}{3} \|x \times y\|^2 I_3 \right) d\varrho(x) d\varrho(y), \qquad \varrho \in \CP(\BBS^{2}),
\end{equation}
where $\times$ denotes the cross product in $\R^3$.
In this case, the multi-particle Oja flow \eqref{eq:MP_oja_flow} becomes
\begin{equation}\label{eq:MP_oja_flow_three_body_kernel}
\dot{x}_k = -\left( I_3 - x_k x_k^{\top} \right) \frac{1}{n^2} \sum_{i,j=1}^n \left( (x_i \times x_j) (x_i \times x_j)^{\top} \right) x_k, \qquad k=1,2, \dots, n \, .
\end{equation}
At the mean-field level, let $\rho(t)$ again denote the solution of the associated continuity equation \eqref{eq:mf_oja_flow} with the interaction kernel $\CF$ defined in \eqref{eq:CF_three_body}.

Applying the ACG reduction in Theorem \ref{thm:ACG_invariant}, if the initial data $\rho_0 = \ACG(\Sigma_0)$ for some $\Sigma_0 \in \SPD_1(3)$, then
\begin{equation*}
\rho(t) = \ACG(\Sigma(t))
\end{equation*}
for all $t \geq 0$, where $\Sigma(t)$ solves
\begin{equation}\label{eq:ACG_param_example_three_body}
    \dot{\Sigma}(t) = 4 \left( M(\Sigma(t)) - M(\Sigma(t))^2 - \frac{1}{3} \left( 1 - \tr{M(\Sigma(t))^2} \right) I_3 \right) \Sigma(t), \qquad \Sigma(0) = \Sigma_0 \, .
\end{equation}
Here, $M(\dummy)$ denotes the second moment matrix of the ACG distribution defined in \eqref{eq:ACG_second_moment}.
A detailed derivation of the reduced equation is provided in Lemma \ref{lem:ACG_param_dyn_example_3}.

The relation between the stationary distributions within the ACG family and those of the original mean-field equation is similar to that observed in Example 2. 
In particular, within the ACG family, the uniform distribution $\ACG(I_3)$ is the unique stationary distribution of the reduced dynamics \eqref{eq:ACG_param_example_three_body}.
The original mean-field dynamics \eqref{eq:mf_oja_flow}, however, possess additional stationary distributions outside the ACG family.\footnote{We omit the derivation here, as it follows an argument similar to that used in Example 2.}

As in Example $2$, the mean-field evolution can also be simulated directly through the reduced equation for $\Sigma(t)$.
Figure \ref{fig:three_body_sigma} shows the resulting evolution of $\rho(t) = \ACG(\Sigma(t))$ on $\BBS^2$, visualized using samples drawn at several successive time points.
The agreement with the large-$n$ particle simulation shown in Figure \ref{fig:oja-three-body} again illustrates the consistency between the finite-particle evolution and the mean-field evolution described by the ACG reduction.

\begin{figure}[!htb]
    \centering
    \includegraphics[width=1.0\linewidth]{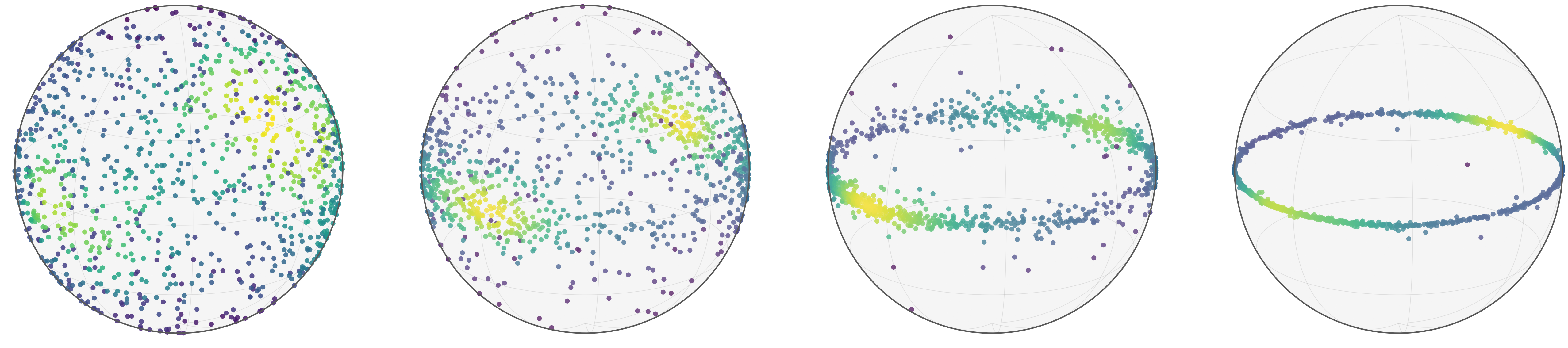}
    \caption{Evolution of the mean-field distribution $\rho(t) = \ACG(\Sigma(t))$ in dimension $d=3$, where $\Sigma(t)$ solves \eqref{eq:ACG_param_example_three_body}.
    The four snapshots visualize the evolving distribution on $\BBS^2$ by displaying samples drawn from $\ACG(\Sigma(t))$ at successive time points.}
    \label{fig:three_body_sigma}
\end{figure}

To the best of our knowledge, a rigorous theoretical characterization of the long-time behavior of the multi-particle Oja flow with the cross-product interaction kernel \eqref{eq:CF_three_body} has not yet been established.
Numerical simulations in \cite{lohe2025exactA} indicate circular synchronization, consistent with the behavior observed above.
As discussed in Remark \ref{rem:long-time_behavior_example_2} for Example 2, the reduced equation \eqref{eq:ACG_param_example_three_body} obtained through the ACG reduction provides a natural route to investigating the qualitative properties and long-time behavior of the associated mean-field dynamics.


\section{Conclusion}\label{sec:conclusion}
In this paper, we systematically study the multi-particle Oja flow \eqref{eq:MP_oja_flow} as a family of interacting particle systems on the unit sphere.
After establishing well-posedness at both the finite-particle and mean-field levels, together with a mean-field convergence result, we explore the low-dimensional structures underlying the dynamics.

In particular, we show that, for an arbitrary initial distribution, the solution of the mean-field equation \eqref{eq:mf_oja_flow} admits a representation via normalized linear transformations, providing an explicit expression for the associated characteristic flow.
Building upon this representation, we further identify a broad class of invariant parametric families for the mean-field dynamics.
When the initial distribution belongs to one of these families, the evolution remains within the same family and can therefore be characterized through the dynamics of the corresponding family parameters.
As a concrete example, we develop this reduction within the angular central Gaussian family and derive the corresponding parameter matrix dynamics.

These results provide a foundation for further investigation of the dynamical properties and collective behavior of multi-particle Oja flows with different choices of the interaction kernel. 
As illustrated by the examples in Section \ref{sec:examples}, reductions via invariant parametric families can be particularly useful when the initial distribution belongs to such a family and the resulting dynamics admit a more tractable description than the original mean-field equation.
In this setting, questions concerning the dynamical behaviors of the mean-field system can be studied through a finite-dimensional dynamical system rather than directly on the space of probability measures. 
Although such a reduction does not capture the mean-field dynamics for arbitrary initial distributions, it provides a useful analytical and computational framework in regimes where direct analysis of the original mean-field equation may be difficult.

More broadly, the general form of the multi-particle Oja flow, together with the low-dimensional structures identified in this paper, provides a flexible framework for exploring interaction kernels motivated by applications beyond those considered here and for analyzing the resulting collective dynamics.

\section*{Acknowledgements}
The work of S. Li was supported by the National Science Foundation under Grant DMS-2511283. 
S. Li would like to sincerely thank Fei Lu, Nicolás García Trillos, and Mauro Maggioni for their helpful discussions and valuable suggestions.

\paragraph{Statement on the use of AI-assisted tools.}
GPT-5.6 Sol was used to assist with proofreading and language editing of the manuscript, as well as with several routine algebraic calculations. 
In particular, it was used to assist with deriving and simplifying the example-specific reduced dynamics presented in Section \ref{sec:examples}, based on the general framework established in Section \ref{sec:low-dim_str}.
It was also used to assist in generating Python scripts for the numerical simulations. 
Following its use, the author independently re-derived and rigorously verified all mathematical statements, proofs, computations, and code. 
The author takes full responsibility for the accuracy, integrity, and originality of the final manuscript.

\bibliographystyle{unsrt}
\bibliography{ref}

\newpage
\appendix
\clearpage
\phantomsection
\addcontentsline{toc}{section}{Appendix}
\addtocontents{toc}{\protect\setcounter{tocdepth}{0}}

\etocdepthtag.toc{appendix}


\section{Proofs and Auxiliary Computations}

\subsection{Proof Details in Section \ref{sec:well-posedness}}\label{app:proof_sec_2}

\begin{proof}[Proof of Lemma \ref{lem:vf_property}]
We first note that, for every $x \in \BBS^{d-1}$, the matrix $\Pperp_x := (I_d - xx^{\top})$ is the orthogonal projection onto the tangent space $\mathrm{T}_{x} \BBS^{d-1}$, and hence $\|\Pperp_x\|_{\op} = 1$.

Recall the vector field $b$ defined in \eqref{eq:vf}.
We first show the uniform boundedness of $b$ for every $x \in \BBS^{d-1}$ and $\varrho \in \CP(\BBS^{d-1})$.
In particular, using the uniform boundedness of $\CF$, we obtain 
\begin{equation*}
\|b(x, \varrho)\| = \left\| \P_x^{\perp} \CF[\varrho] x \right\| \leq \left\| \P_x^{\perp} \right\|_{\op} \left\| \CF[\varrho] \right\|_{\op} \|x\| \leq C_{\CF} \, .
\end{equation*}

We next establish the Lipschitz estimate of $b$. 
Let $x, \Wx \in \BBS^{d-1}$ and $\varrho, \Wvarrho \in \CP(\BBS^{d-1})$, and abbreviate $F := \CF[\varrho]$ and $\WF := \CF[\Wvarrho]$.
By the triangle inequality, we have
\begin{equation*}
\begin{aligned}
\left\| b(x, \varrho) - b(\Wx, \Wvarrho) \right\| &= \left\|\P_x^{\perp} F x - \P_{\Wx}^{\perp} \WF \Wx \right\| \leq \left\| \left(\Pperp_x - \Pperp_{\Wx}\right) Fx \right\| + \left\| \Pperp_{\Wx} \left( F - \WF \right)x \right\| + \left\| \Pperp_{\Wx} \WF \left( x - \Wx \right) \right\|\\[4pt]
&\leq \left\| \Pperp_x - \Pperp_{\Wx} \right\|_{\op} \|F\|_{\op} \|x\| + \left\|\Pperp_{\Wx} \right\|_{\op} \left\| F - \WF \right\|_{\op} \|x\| + \left\| \Pperp_{\Wx} \right\|_{\op} \left\| \WF \right\|_{\op} \left\| x - \Wx \right\|\\[4pt]
&\leq 2 C_{\CF} \left\| x - \Wx \right\| + L_{\CF} W_1 (\varrho, \Wvarrho) + C_{\CF} \left\| x - \Wx \right\| = 3C_{\CF} \left\| x - \Wx \right\| + L_{\CF} W_1 (\varrho, \Wvarrho) \, .
\end{aligned}
\end{equation*}
Here, the last inequality uses the boundedness and Lipschitzness of $\CF$, and $\|\Pperp_{\Wx}\|_{\op} = 1$.
\end{proof}

\subsection{More on the Angular Central Gaussian Family}\label{app:ACG}

\begin{proof}[Proof of Proposition \ref{prop:stationary_states}]
We first characterize the stationary distributions of the original mean-field equation \eqref{eq:mf_oja_flow}.
Let $\rho_*$ be a stationary distribution of \eqref{eq:mf_oja_flow}.
Then, for every test function $\varphi \in C^1(\BBS^{d-1})$, 
\begin{equation*}
\int_{\BBS^{d-1}} \nabla_{\BBS^{d-1}} \varphi (x)^{\top} \left( I_d - x x^{\top} \right) \CF[\rho_*] x d \rho_*(x) = 0 \, .
\end{equation*}
Choose $\varphi(x) = \frac{1}{2} x^{\top} \CF[\rho_*] x$.
Since $\CF[\rho_*]$ is symmetric, $\nabla_{\BBS^{d-1}} \varphi(x) = \left( I_d - x x^{\top} \right)\CF[\rho_*] x$.
Therefore,
\begin{equation*}
\int_{\BBS^{d-1}} \left\| \left( I_d - x x^{\top} \right) \CF[\rho_*] x  \right\|^2 d\rho_*(x) = 0 \, .
\end{equation*}
Since the integrand is nonnegative, it follows that
\begin{equation}\label{eq:stat_cond_mf}
    \left( I_d - x x^{\top} \right) \CF[\rho_*] x = 0 \qquad \rho_*\text{-a.e.} \, .
\end{equation}
Conversely, if a probability measure $\rho_*$ satisfies \eqref{eq:stat_cond_mf}, then the velocity field $b$, as defined in \eqref{eq:vf}, vanishes $\rho_*$-almost everywhere. 
Hence $\rho_*$ is stationary.
Therefore, the set of stationary distributions of the mean-field equation \eqref{eq:mf_oja_flow} is precisely $\CS$, as defined in \eqref{eq:stat_state_mf}.

We next characterize the stationary distributions within the angular central Gaussian family.
Consider the mean-field equation \eqref{eq:mf_oja_flow} with initial data $\rho_0 = p_{\Sigma_0}$ for some $\Sigma_0 \in \SPD_1(d)$.
By Theorem \ref{thm:ACG_invariant}, the corresponding solution remains within the ACG family and satisfies $\rho(t) = p_{\Sigma(t)}$, where $\Sigma(t)$ solves the matrix equation \eqref{eq:mf_goven_dyn_acg}.

Now suppose that $\rho_* = p_{\Sigma_*}$ is a stationary distribution within the ACG family.
For notational simplicity, set $\CA_* := \CF[p_{\Sigma_*}]$.
Since $\CF$ is symmetric, the stationary condition for $\Sigma_*$ becomes 
\begin{equation}\label{eq:stat_cond_Sigma_origin}
0 = \dot{\Sigma}_* = \CA_* \Sigma_* + \Sigma_* \CA_* \, .
\end{equation}
Since $\Sigma_*$ is symmetric positive definite, it admits the eigendecomposition
\[
\Sigma_* = Q D Q^{\top},
\]
where $Q$ is orthogonal and $D = \diag (\lambda_1, \lambda_2, \dots, \lambda_d)$ with $\lambda_i > 0$.
Define
\[
\WCA_* := Q^{\top} \CA_* Q \, .
\]
Multiplying \eqref{eq:stat_cond_Sigma_origin} from the left by $Q^{\top}$ and from the right by $Q$, we obtain
\begin{equation*}
    \WCA_* D + D \WCA_* = 0 \, .
\end{equation*}
Hence, for every $i, j = 1,\dots, d$,
\begin{equation*}
    \left( \lambda_i + \lambda_j \right) (\WCA_{*})_{ij} = 0 \, .
\end{equation*}
Since $\lambda_i + \lambda_j > 0$, we obtain $(\WCA_*)_{ij} = 0$ for all $i, j$.
Therefore, $\WCA_* = 0$, and consequently
\begin{equation*}
    \CA_* = \CF[p_{\Sigma_*}] = 0 \, .
\end{equation*}
Conversely, if $\CF[p_{\Sigma_*}] = 0$, then the vector field $b$ vanishes $p_{\Sigma_*}$-almost everywhere.
Hence $p_{\Sigma_*}$ is stationary.
Therefore, the set of stationary distributions of \eqref{eq:mf_oja_flow} within the ACG family is precisely $\CS_{\ACG}$, as defined in \eqref{eq:stat_state_ACG}.

Finally, the inclusion $\CS_{\ACG} \subseteq \CS$ follows immediately from the definitions of the two sets. 
This completes the proof.
\end{proof}

\begin{proposition}[A characterization of the angular central Gaussian family]\label{prop:M_Unif_ACG_manifold}
Denote the uniform distribution on $\BBS^{d-1}$ by $\Unif(\BBS^{d-1})$, or simply $\Unif$.
Then,
\begin{equation}\label{eq:ACG_manifold}
\{(\Phi_G)_{\sharp} \Unif: G \in \SLpm(d,\R)\} = \left\{\ACG(\Sigma): \Sigma \in \SPD_1(d)   \right\} \, .
\end{equation}
\end{proposition}

\begin{proof}[Proof of Proposition \ref{prop:M_Unif_ACG_manifold}]
Let $Z \sim \CN (0, I_d)$ be a standard Gaussian random vector in $\R^d$.
Then $U := \frac{Z}{\|Z\|}$ follows the uniform distribution over $\BBS^{d-1}$, namely, $U \sim \Unif(\BBS^{d-1})$.
Fix $G \in \SLpm(d, \R)$, and define
\begin{equation*}
    X = \Phi_G(U) \, .
\end{equation*}
By the definition of the normalized linear transformation \eqref{eq:NL_transform}, $X = \frac{GU}{\|GU\|} = \frac{GZ}{\|GZ\|}$.
Setting $Y := GZ$ and $\Sigma := GG^{\top}$, we have
\begin{equation*}
   Y \sim \CN (0, \Sigma) \qquad \text{and} \qquad  X = \frac{Y}{\|Y\|} \, .
\end{equation*}
A standard polar-coordinate change-of-variables argument shows that the direction of the centered Gaussian random vector $Y$ follows the angular central Gaussian distribution with parameter matrix $\Sigma$; see, for example, \cite{tyler1987statistical}.
Therefore, $X \sim (\Phi_G)_{\sharp} \Unif = \ACG (\Sigma)$.
Moreover, since $G \in \SLpm(d, \R)$, we know that $\Sigma \in \SPD_1(d)$.
This proves 
\begin{equation*}
    \{(\Phi_G)_{\sharp} \Unif: G \in \SLpm(d,\R)\} \subseteq \left\{\ACG(\Sigma): \Sigma \in \SPD_1(d) \right\} \, .
\end{equation*}

Conversely, let $\Sigma \in \SPD_1(d)$.
Since $\Sigma$ is symmetric positive definite, it admits a unique symmetric positive definite square root $\Sigma^{1/2}$.
Set $G := \Sigma^{1/2}$; then $G$ satisfies $GG^{\top} = \Sigma$ and $\det G = \sqrt{\det \Sigma} = 1$.
Therefore, $G \in \SLpm(d, \R)$.
Applying the preceding argument with $G = \Sigma^{1/2}$ gives $\ACG(\Sigma) = (\Phi_{G})_{\sharp} \Unif$.
Therefore,
\begin{equation*}
    \left\{ \ACG(\Sigma) : \Sigma \in \SPD_1(d) \right\} \subseteq \{(\Phi_G)_{\sharp} \Unif: G \in \SLpm(d,\R)\} \, .
\end{equation*}
Combining the two inclusions proves the result.
\end{proof}

\subsection{Auxiliary Computations in Section \ref{sec:examples}}\label{app:examples_aux}

\begin{lemma}[Asymptotic behavior of $\Sigma(t)$ in Example 1]\label{lem:asm_behavior_Sigma_t}
Let $B \in \R^{d\times d}$ be symmetric with a unique largest eigenvalue, and let $v_1$ be a unit eigenvector corresponding to this eigenvalue.
Then the matrix $\Sigma(t)$ defined in \eqref{eq:ACG_Sigma_classical_Oja_flow} with $\Sigma_0 \in \SPD_1(d)$ satisfies
\begin{equation*}
    \frac{\Sigma(t)}{\tr{\Sigma(t)}} \rightarrow v_1 v_1^{\top}, \qquad \text{as } t \rightarrow \infty \, .
\end{equation*}
\end{lemma}

\begin{proof}[Proof of Lemma \ref{lem:asm_behavior_Sigma_t}]
Let $\{(\lambda_i, v_i)\}_{i=1}^d$ be the eigenpairs of $B$, ordered so that $\lambda_1 > \lambda_2 \geq \dots \geq \lambda_d$, and let $\{v_i\}_{i=1}^d$ be an orthonormal basis of eigenvectors.
Since $B$ is symmetric,
\begin{equation*}
    e^{tB} = \sum_{i=1}^d e^{\lambda_i t} v_i v_i^{\top} \, .
\end{equation*}
Hence,
\begin{equation*}
\Sigma(t) = e^{tB} \Sigma_0 e^{tB} = \sum_{i,j=1}^d e^{(\lambda_i + \lambda_j)t} (v_i^{\top} \Sigma_0 v_j) v_i v_j^{\top} \, .
\end{equation*}
Taking the trace and using $\tr{v_i v_j^{\top}} = v_j^{\top} v_i = \delta_{ij}$, we obtain
\begin{equation*}
\tr{\Sigma(t)} =  \sum_{i=1}^d e^{2\lambda_i t} \left(v_i^{\top} \Sigma_0 v_i \right) \, .
\end{equation*}
Therefore,
\begin{equation*}
    \frac{\Sigma(t)}{\tr{\Sigma(t)}} = \frac{\sum_{i,j=1}^d e^{(\lambda_i + \lambda_j - 2\lambda_1)t} (v_i^{\top} \Sigma_0 v_j) v_i v_j^{\top}}{\sum_{i=1}^d e^{2(\lambda_i - \lambda_1) t} \left(v_i^{\top} \Sigma_0 v_i \right)} \, .
\end{equation*}
Since $\lambda_1$ is the unique largest eigenvalue and $\Sigma_0$ is positive definite, it follows that
\begin{equation*}
\frac{\Sigma(t)}{\tr{\Sigma(t)}}   \rightarrow \frac{(v_1^{\top} \Sigma_0 v_1) v_1 v_1^{\top}}{v_1^{\top} \Sigma_0 v_1} =  v_1 v_1^{\top}, \qquad \text{as } t \rightarrow \infty \, .
\end{equation*}
\end{proof}

\begin{lemma}[Stationary distribution within the ACG family in Example 2]\label{lem:stat_dis_example_2}
Let $M(\Sigma)$ denote the second moment of the angular central Gaussian distribution $\ACG(\Sigma)$ with parameter $\Sigma \in \SPD_1(d)$, as given in \eqref{eq:ACG_second_moment}.
Then 
\begin{equation*}
\left\{ \ACG(\Sigma_*): M(\Sigma_*) = \frac{1}{d} I_d, \;\; \Sigma_* \in \SPD_1(d) \right\} = \left\{ \ACG(I_d) \right\} \, .
\end{equation*}
\end{lemma}

\begin{proof}
Let $\Sigma_* \in \SPD_1(d)$.
We first establish a useful property of the second moment $M(\Sigma_*)$.
Since $\Sigma_*$ is symmetric positive definite, it admits the eigendecomposition
\[
\Sigma_* = Q D Q^{\top},
\]
where $Q$ is orthogonal and
\[
D = \diag \left( \lambda_1, \dots, \lambda_d \right), \qquad \lambda_i > 0 \, .
\]
Using the integral representation of $M(\Sigma)$ given in \eqref{eq:ACG_second_moment_main}, together with $I_d + s\Sigma_* = Q (I_d + s D) Q^{\top}$ for every $s \in [0,\infty)$, we obtain
\begin{equation*}
Q^{\top} M(\Sigma_*) Q =  \frac{1}{2} \int_0^{\infty} \frac{D (I_d + s D)^{-1}}{\sqrt{\det (I_d + sD)}} ds \, . 
\end{equation*}
Therefore, $M(\Sigma_*)$ and $\Sigma_*$ have the same eigenvectors.
Denote the corresponding eigenvalues of $M(\Sigma_*)$ by
\begin{equation}\label{eq:miStar}
    m_i := \frac{1}{2} \int_0^{\infty} \frac{\lambda_i (1 + s\lambda_i)^{-1}}{\sqrt{\prod_{k=1}^d (1 + s\lambda_k)}} ds, \qquad i = 1, \dots, d \, .
\end{equation}
Then $Q^{\top} M(\Sigma_*) Q = \diag(m_1, \dots, m_d)$.
For any $i,j = 1,\dots,d$, 
\begin{equation*}
    m_i - m_j = \frac{1}{2} \int_0^{\infty} \frac{1}{\sqrt{\prod_{k=1}^d (1 + s\lambda_k)}} \frac{\lambda_i - \lambda_j}{(1 + s\lambda_i) (1 + s\lambda_j)} ds \, .
\end{equation*}
Since $\lambda_i, \lambda_j > 0$, we know that $m_i - m_j$ has the same sign as $\lambda_i - \lambda_j$.
Moreover, $m_i = m_j$ if and only if $\lambda_i = \lambda_j$.

Now suppose that $M(\Sigma_*) = \frac{1}{d} I_d$.
Then
\[
m_i = \frac{1}{d}, \qquad i = 1, \dots, d \, .
\]
Thus $m_i = m_j$ for all $i,j$, and the previous observation implies that 
\[
\lambda_1 = \lambda_2 = \cdots = \lambda_d =: \lambda_* \, .
\]
Since $\Sigma_* \in \SPD_1(d)$, we have
\[
1 = \det(\Sigma_*) = \prod_{i=1}^d \lambda_i = (\lambda_*)^d \, .
\]
Because $\lambda_* > 0$, it follows that $\lambda_* = 1$.
Therefore, $\Lambda_* = I_d$, and hence $\Sigma_* = Q I_d Q^{\top} = I_d$.
This proves that $M(\Sigma_*) = \frac{1}{d} I_d$ implies $\Sigma_* = I_d$.

Conversely, suppose that $\Sigma_* = I_d$.
Then, by the integral representation of $M$ in \eqref{eq:ACG_second_moment_main},
\[
M(I_d) = \left(\frac{1}{2} \int_0^{\infty} (1 + s)^{-\frac{d}{2} - 1} ds\right) I_d = \frac{1}{d} I_d \, .
\]
Therefore, 
\begin{equation*}
\left\{ \ACG(\Sigma_*): M(\Sigma_*) = \frac{1}{d} I_d, \Sigma_* \in \SPD_1(d) \right\} = \left\{ \ACG(I_d) \right\} \, .
\end{equation*}
This completes the proof.
\end{proof}

\begin{lemma}[Angular central Gaussian reduction for Example $3$]\label{lem:ACG_param_dyn_example_3}
Let $\rho(t)$ be the solution of the mean-field equation \eqref{eq:mf_oja_flow} with the cross-product interaction kernel $\CF$ defined in \eqref{eq:CF_three_body}.
If $\rho_0 = \ACG(\Sigma_0)$ for some $\Sigma_0 \in \SPD_1 (3)$, then $\rho(t) = \ACG (\Sigma(t))$ for all $t \geq 0$, where $\Sigma(t)$ solves the ODE \eqref{eq:ACG_param_example_three_body}.
\end{lemma}

\begin{proof}[Proof of Lemma \ref{lem:ACG_param_dyn_example_3}]
By applying the ACG reduction established in Theorem \ref{thm:ACG_invariant}, if $\rho_0 = \ACG(\Sigma_0)$, then
\[
\rho(t) = \ACG (\Sigma(t))
\]
for all $t \geq 0$, where $\Sigma(t)$ satisfies
\begin{equation*}
    \dot{\Sigma}(t) = \CF [p_{\Sigma(t)}] \Sigma(t) + \Sigma(t) \CF [p_{\Sigma(t)}]^{\top}, \qquad \Sigma(0) = \Sigma_0 \, .
\end{equation*}
Let $M(\Sigma(t))$ denote the second moment matrix of $p_{\Sigma(t)}$ as in \eqref{eq:ACG_second_moment}.
Applying Lemma \ref{lem:aux_cross_product}, we obtain
\begin{equation*}
\begin{aligned}
\CF [p_{\Sigma(t)}] &= -\int_{\BBS^2} \int_{\BBS^2} \left( (x \times y) (x \times y)^{\top} - \frac{1}{3} \|x \times y\|^2 I_3 \right) p_{\Sigma(t)}(x) p_{\Sigma(t)}(y) dx dy\\
&= - \left( 2 M(\Sigma(t))^2 - 2M(\Sigma(t)) + \left( 1 - \tr{M(\Sigma(t))^2} \right) I_3 \right) + \frac{1}{3} (1 - \tr{M(\Sigma(t))^2}) I_3\\
&= 2 \left(M(\Sigma(t)) - M(\Sigma(t))^2 - \frac{1}{3} \left( 1 - \tr{M(\Sigma(t))^2} \right) I_3 \right) \, .
\end{aligned}
\end{equation*}
Since $\CF[p_{\Sigma(t)}]$ is symmetric, substituting the above expression into the evolution equation for $\Sigma(t)$ and using the fact that $M(\Sigma(t))$ commutes with $\Sigma(t)$ gives \eqref{eq:ACG_param_example_three_body}, completing the proof.
\end{proof}

\begin{lemma}\label{lem:aux_cross_product}
Let $\varrho \in \CP(\BBS^2)$, and let $X, Y \overset{\text{i.i.d.}}{\sim} \varrho$.
Define
\begin{equation*}
    M_{\varrho} := \E [X X^{\top}], \qquad K_{\varrho} := \E \left[ (X \times Y) (X \times Y)^{\top} \right] \, .
\end{equation*}
Then,
\begin{equation*}
    K_{\varrho} = 2\adj(M_{\varrho}) = 2 M_{\varrho}^2 - 2 M_{\varrho} + \left( 1 - \tr{M_{\varrho}^2} \right) I_3,
\end{equation*}
and
\begin{equation*}
    \tr{K_{\varrho}} = 1 - \tr{M_{\varrho}^2} \, .
\end{equation*}
\end{lemma}

\begin{proof}
Fix $\varrho \in \CP(\BBS^2)$.
For notational simplicity, we suppress the dependence on $\varrho$ and write $M := M_{\varrho}$ and $K:= K_{\varrho}$.

Since $M$ is real and symmetric, there exists an orthogonal matrix $Q \in \R^{3 \times 3}$ with $\det Q = 1$ such that
\begin{equation*}
    M = Q D Q^{\top}, \qquad D:= \diag(\lambda_1, \lambda_2, \lambda_3),
\end{equation*}
where $\lambda_1,\lambda_2, \lambda_3$ are the eigenvalues of $M$.

Define the rotated random vectors
\begin{equation*}
\WX = Q^{\top} X, \qquad \WY = Q^{\top} Y \, .
\end{equation*}
Since $X$ and $Y$ are independent and identically distributed, so are $\WX$ and $\WY$.
Moreover,
\begin{equation}\label{eq:WX_WY_second_moment}
    \E \left[ \WX \WX^{\top} \right] = \E \left[ \WY \WY^{\top} \right] = \E \left[ Q^{\top} X X^{\top} Q \right] = Q^{\top} M Q = D \, .
\end{equation}
Set 
\begin{equation*}
    \WK := \E \left[ (\WX \times \WY) (\WX \times \WY)^{\top} \right] \, .
\end{equation*}
We first compute $\WK \in \R^{3 \times 3}$.
With a slight abuse of notation, we write
\begin{equation*}
    \WX = \left( X_1, X_2, X_3 \right)^{\top}, \qquad \WY = (Y_1, Y_2, Y_3)^{\top} \, .
\end{equation*}
By \eqref{eq:WX_WY_second_moment}, we have
\begin{equation*}
\E [X_i X_j] = \E [Y_i Y_j] = \delta_{ij} \lambda_i, \qquad i, j = 1,2,3 \, .
\end{equation*}

Recall that
\begin{equation*}
 \WX \times \WY = \left( \begin{array}{c}
    X_2 Y_3 - X_3 Y_2   \\
    X_3 Y_1 - X_1 Y_3 \\
    X_1 Y_2 - X_2 Y_1
 \end{array} \right) \, .
\end{equation*}
For the first diagonal entry of $\WK$, independence of $\WX$ and $\WY$ gives
\begin{equation*}
\WK_{11} = \E \left[ (X_2 Y_3 - X_3 Y_2)^2 \right] = \E [X_2^2] \E [Y_3^2] + \E [X_3^2] \E [Y_2^2] - 2 \E [X_2 X_3] \E [Y_2 Y_3] = 2 \lambda_2 \lambda_3 \, .
\end{equation*}
Similarly,
\[
\WK_{22} = 2\lambda_1 \lambda_3, \qquad \WK_{33} = 2\lambda_1 \lambda_2 \, .
\]
We can further show that the off-diagonal entries vanish.
For example,
\[
\WK_{12} = \E \left[ (X_2 Y_3 - X_3 Y_2) (X_3 Y_1 - X_1 Y_3) \right] \, .
\]
After expanding, each term contains an off-diagonal second moment of either $\WX$ or $\WY$, which vanishes because their second moment matrix is diagonal.
The same argument applies to the remaining off-diagonal entries.
Therefore,
\begin{equation*}
    \WK = 2 \left( \begin{array}{ccc}
        \lambda_2 \lambda_3 & 0 & 0  \\
        0 & \lambda_1 \lambda_3 & 0 \\
        0 & 0 & \lambda_1 \lambda_2
    \end{array} \right) = 2 \adj (D) \, .
\end{equation*}

Since $Q$ is an orthogonal matrix satisfying $\det (Q) = 1$, the cross product is equivariant under $Q$:
\begin{equation*}
    (Q^{\top} X) \times (Q^{\top} Y) = Q^{\top} (X \times Y) \, .
\end{equation*}
Therefore,
\begin{equation*}
\WK = \E \left[ (\WX \times \WY) (\WX \times \WY)^{\top} \right] = Q^{\top} \E \left[ (X \times Y) (X \times Y)^{\top} \right] Q = Q^{\top} K Q \, .
\end{equation*}
It follows that
\begin{equation*}
K = Q \WK Q^{\top} = 2 Q \adj (D)  Q^{\top} \, .
\end{equation*}
For an arbitrary $3 \times 3$ matrix $A$, the Cayley--Hamilton theorem \cite{horn2012matrix} gives
\begin{equation*}
    \adj (A) = A^2 - \tr{A} A + \frac{1}{2} \left( (\tr{A})^2 - \tr{A^2} \right) I_3 \, .
\end{equation*}
Applying this identity to $Q D Q^{\top}$ gives
\begin{equation*}
\begin{aligned}
\adj (Q D Q^{\top})
&= Q D^2 Q^{\top} - Q \left( \tr{D} D \right) Q^{\top} + \frac{1}{2} \left( (\tr{D})^2 - \tr{D^2} \right) I_3\\
&= Q \left( D^2 - \tr{D} D + \frac{1}{2} \left( (\tr{D})^2 - \tr{D^2} \right) I_3 \right) Q^{\top} = Q \adj (D) Q^{\top} \, .
\end{aligned}
\end{equation*}
Since $M = Q D Q^{\top}$, we conclude that
\begin{equation*}
    K = 2 \adj (M) \, .
\end{equation*}
Finally, because $\varrho$ is supported on $\BBS^2$,
\begin{equation*}
    \tr{M} = \E \left[ \tr{X X^{\top}} \right] = \E [\|X\|^2] = 1 \, .
\end{equation*}
Hence,
\[
K = 2M^2 - 2 M + \left( 1 - \tr{M^2} \right) I_3 \, .
\]
Taking the trace gives
\begin{equation*}
\tr{K} = 2\tr{M^2} - 2\tr{M} + (1 - \tr{M^2}) \tr{I_3} = 1 - \tr{M^2} \, .
\end{equation*}
This completes the proof.
\end{proof}

\end{document}